\documentclass[12pt]{scrartcl}    
\usepackage{a4} 
\usepackage{amsmath}    
\usepackage{paralist}
\usepackage{amssymb}  
\usepackage{amsfonts}  
\usepackage{mathrsfs}  
\usepackage{dsfont}
\usepackage{latexsym} 
\usepackage{xcolor}
\usepackage{bbm,exscale}
\definecolor{Myblue}{rgb}{0,0,0.6}  
\usepackage[colorlinks,citecolor=Myblue,linkcolor=Myblue,urlcolor=Myblue,pdfpagemode=None]{hyperref}
\usepackage{amsthm}
\usepackage{accents}
\usepackage[square,numbers,sort&compress]{natbib} 
\usepackage[all,cmtip]{xy}
\usepackage{ifthen} 
\usepackage{bbding}
\usepackage{stmaryrd}  
\usepackage{wasysym}
\usepackage{verbatim}
\usepackage{bbding} 
\usepackage{soul}  
\usepackage[yyyymmdd,hhmmss]{datetime}
\usepackage{booktabs}
\usepackage{enumitem}
\usepackage{color}
\usepackage{textcomp}
\usepackage{gensymb}
\usepackage{pdfpages}

\vfuzz \hfuzz
\makeatletter
\newcommand{\raisemath}[1]{\mathpalette{\raisem@th{#1}}}
\newcommand{\raisem@th}[3]{\raisebox{#1}{$#2#3$}}
\makeatother

\newcommand{\B}{\mathcal{B}}
\newcommand{\Bfd}{\mathcal{B}^{\textrm{fd}}}

\newcommand{\C}{\mathds{C}}

\newcommand{\Q}{\mathds{Q}}
\newcommand{\R}{\mathds{R}}
\newcommand{\Z}{\mathds{Z}}

\def\1{\ifmmode\mathrm{1\!l}\else\mbox{\(\mathrm{1\!l}\)}\fi}

\newcommand{\be}{\begin{equation}}
  \newcommand{\ee}{\end{equation}}
\newcommand{\bes}{\begin{equation*}}
  \newcommand{\ees}{\end{equation*}}

\newcommand{\id}{\text{id}}

\newcommand{\Hom}{\operatorname{Hom}}
\newcommand{\End}{\operatorname{End}}

\def\LG{\mathcal{LG}}
\def\LGgr{\mathcal{LG}^{\mathrm{gr}}}
\def\LGhuge{\mathcal{LG}^{\mathrm{GR}}}

\newcommand{\idlg}{\textrm{Id}_{\LG}}
\newcommand{\hmf}{\operatorname{hmf}}

\newcommand{\ev}{\operatorname{ev}}

\newcommand{\tev}{\widetilde{\operatorname{ev}}}
\newcommand{\coev}{\operatorname{coev}}
\newcommand{\tcoev}{\widetilde{\operatorname{coev}}}
\def\lra{\longrightarrow}

\def\lmt{\longmapsto}

\DeclareMathOperator{\Jac}{Jac}

\DeclareMathOperator{\Res}{Res}

\newcommand{\Bord}{\operatorname{Bord}}
\newcommand{\Bordfr}{\Bord_{2,1,0}^{\textrm{fr}}}
\newcommand{\Bordor}{\Bord_{2,1,0}^{\textrm{or}}}
\newcommand{\zzwfr}{\zz_W^{\textrm{fr}}}
\newcommand{\zzwor}{\zz_W^{\textrm{or}}}
\newcommand{\sotwo}{\textrm{SO}(2)}
\newcommand{\kbso}{\mathscr{K}(\Bfd)^{\sotwo}}
\newcommand{\zz}{\mathcal{Z}}

\newcommand{\Vect}{\operatorname{Vect}}
\newcommand{\Vectk}{\operatorname{Vect}_\Bbbk}

\newcommand{\sta}{\boxempty}

\newcommand{\dX}{{}^\dagger\hspace{-1.8pt}X}

\newcommand\arxiv[2]      {\href{https://arXiv.org/abs/#1}{#2}}
\newcommand\doi[2]        {\href{https://dx.doi.org/#1}{#2}}

\allowdisplaybreaks

\deffootnote[1em]{1em}{1em}{\textsuperscript{\thefootnotemark}}

\theoremstyle{definition} 
\newtheorem{definition}{Definition}
\newtheorem{proposition}[definition]{Proposition}
\newtheorem{theorem}[definition]{Theorem}

\newtheorem{lemma}[definition]{Lemma}
\newtheorem{corollary}[definition]{Corollary}
\newtheorem{remark}[definition]{Remark}

\newtheorem{example}[definition]{Example}

\numberwithin{equation}{section}
\numberwithin{definition}{section}
\numberwithin{figure}{section}

\newcommand\void[1]{}

\begin{document}

\title{%
Extending Landau-Ginzburg models to the point%
}

\author{%
  Nils Carqueville \quad
  Flavio Montiel Montoya%
  \\[0.5cm]
 \normalsize{\texttt{\href{mailto:nils.carqueville@univie.ac.at}{nils.carqueville@univie.ac.at}}} \\  %
   \normalsize{\texttt{\href{mailto:montielmof98@univie.ac.at}{montielmof98@univie.ac.at}}}\\[0.1cm]
  \hspace{-1.2cm} {\normalsize\slshape Fakult\"at f\"ur Mathematik, Universit\"at Wien, Austria}\\[-0.1cm]
}

\date{}
\maketitle

\begin{abstract}
We classify framed and oriented 2-1-0-extended TQFTs with values in the bicategories of Landau-Ginzburg models, whose objects and 1-morphisms are isolated singularities and (either $\Z_2$- or $(\Z_2 \times \Q)$-graded) matrix factorisations, respectively. 
For this we present the relevant symmetric monoidal structures and find that every object $W\in\Bbbk[x_1,\dots,x_n]$ determines a framed extended TQFT. 
We then compute the Serre automorphisms~$S_W$ to show that~$W$ determines an oriented extended TQFT if the associated category of matrix factorisations is $(n-2)$-Calabi-Yau. 

The extended TQFTs we construct from~$W$ assign the non-separable Jacobi algebra of~$W$ to a circle. 
This illustrates how non-separable algebras can appear in 2-1-0-extended TQFTs, and more generally that the question of extendability depends on the choice of target category. 
As another application, we show how the construction of the extended TQFT based on $W=x^{N+1}$ given by Khovanov and Rozansky can be derived directly from the cobordism hypothesis. 
\end{abstract}

\newpage

\tableofcontents

\section{Introduction}

Fully extended topological quantum field theory is simultaneously an attempt to capture the quantum field theoretic notion of locality in a simplified rigorous setting, and a source of functorial topological invariants. 
In dimension~$n$, such TQFTs have been formalised as symmetric monoidal $(\infty,n)$-functors from certain categories of bordisms with extra geometric structure to some symmetric monoidal $(\infty,n)$-category~$\mathcal C$. 

The fact that such functors must respect structure and relations among bordisms of all dimensions from~0 to~$n$ is highly restrictive. 
Specifically, the cobordism hypothesis of \cite{BDpaper} as formalised in \cite{l0905.0465, AyalaFrancis2017CH} states that (in the case of bordisms with framings) a TQFT is already determined by what it assigns to the point, and that fully extended TQFTs with values in~$\mathcal{C}$ are equivalent to fully dualisable objects in~$\mathcal C$. 
This is a strong finiteness condition. 
Similar relations hold for bordisms with other types of tangential structures; for example, fully extended TQFTs on oriented bordisms are argued to be described by homotopy fixed points of an induced $\text{SO}(n)$-action on fully dualisable objects in~$\mathcal{C}$. 

\medskip 

In the present paper we are concerned with fully extended TQFTs in dimension $n=2$. 
Following \cite{spthesis, Pstragowski} we take an \textsl{extended framed} (or \textsl{oriented}) \textsl{2-dimensional TQFT with values in a symmetric monoidal bicategory~$\B$} (where~$\B$ is called the \textsl{target}) to be a symmetric monoidal 2-functor 
\be 
\zz \colon \Bord_{2,1,0}^\sigma \lra \B 
\ee 
where $\sigma = \textrm{fr}$ (or $\sigma = \textrm{or}$), without any mention of $\infty$-categories. 
The bicategories $\Bord_{2,1,0}^\sigma$ of points, 1-manifolds with boundary and 2-manifolds with corners (all with structure~$\sigma$) are constructed in detail in \cite{spthesis, Pstragowski}. 
Moreover, these authors prove versions of the cobordism hypothesis (as we briefly review in Section~\ref{sec:extendedTQFTs}), and the relevant $\sotwo$-homotopy fixed points were described in \cite{HSV, HV, Hesse}. 

The example for the target~$\B$ that is dominant in the literature is the bicategory $\text{Alg}_\Bbbk$ (or one of its variants, cf.\ \cite[App.\,A]{bdspv1509.06811}) of finite-dimensional $\Bbbk$-algebras, finite-dimensional bimodules and bimodule maps, where~$\Bbbk$ is some field. 
Using the cobordism hypothesis one finds that extended framed TQFTs with values in $\text{Alg}_\Bbbk$ are classified by finite-dimensional separable $\Bbbk$-algebras \cite{l0905.0465, spthesis}, while in the oriented case the classification is in terms of separable symmetric Frobenius $\Bbbk$-algebras \cite{HSV}. 

On the other hand, non-separable algebras arise prominently in (non-extended) TQFTs. 
Recall e.\,g.\ from \cite{Kockbook} that such TQFTs $\zz_{\textrm{ne}} \colon \Bord_{2,1}^{\textrm{or}} \to \mathcal V$ are equivalent to commutative Frobenius algebras in~$\mathcal V$, where~$\mathcal V$ is a symmetric monoidal 1-category. 
Important examples are the categories of vector spaces, possibly with a $\Z_2$- or $\Z$-grading. 
In $\mathcal V = \Vectk^{\Z_2}$ or $\mathcal V = \Vectk^{\Z}$, Dolbeault cohomologies of Calabi-Yau manifolds serve as examples of non-separable commutative Frobenius algebras (describing B-twisted sigma models). 
Another class of examples of generically non-separable Frobenius algebras (in $\Vectk$) are the Jacobi algebras $\Bbbk[x_1,\dots,x_n]/(\partial_{x_1}W, \dots, \partial_{x_n}W)$ of isolated singularities described by polynomials~$W$. 
The associated TQFTs are Landau-Ginzburg models with potential~$W$. 

Hence we are confronted with the following question: How do sigma models and Landau-Ginzburg models (and other non-extended TQFTs with non-separable Frobenius algebras) relate to fully extended TQFTs? 

\medskip 

A non-extended 2-dimensional TQFT~$\zz_{\textrm{ne}} \colon \Bord_{2,1}^\sigma \to \mathcal V$ can be \textsl{extended to the point} if there is a symmetric monoidal bicategory~$\B$ and an extended TQFT $\zz \colon \Bord_{2,1,0}^\sigma \to \B$ such that (with $\mathds{I}_\B \in \B$ the unit object, and $\emptyset = \mathds{I}_{\Bord_{2,1,0}^\sigma}$)
\be 
\mathcal V \cong \End_\B(\mathds{I}_\B) 
\quad \text{ and } \quad 
\zz_{\textrm{ne}} \cong \zz\Big|_{\End_{\Bord_{2,1,0}^\sigma}(\emptyset)} \, . 
\ee  
Clearly an extension, if it exists, is not unique, as it depends on the target~$\B$. 

We expect that the extendability of the known classes of non-separable TQFTs is captured by the following motto:
\begin{itemize}[leftmargin=0.25cm, rightmargin=0.25cm]
\item[] 
\textsl{%
``If a non-extended 2-dimensional TQFT $\zz_{\operatorname{ne}}$ is a restriction of an appropriate defect TQFT $\zz_{\operatorname{ne}}^{\operatorname{def}}$, then $\zz_{\operatorname{ne}}$ can be extended to the point (at least as a framed theory), with the bicategory $\B_{\zz_{\operatorname{ne}}^{\operatorname{def}}}$ associated to $\zz_{\operatorname{ne}}^{\operatorname{def}}$ as target.''%
}
\end{itemize}

Let us unpack this statement and give concrete meaning to it. 
A 2-dimensional defect TQFT is a symmetric monoidal functor $\zz_{\textrm{ne}}^{\textrm{def}}$ on a category of stratified and decorated oriented 2-bordisms, see \cite{dkr1107.0495, CRS1} or the review \cite{TQFTlecturenotes}. 
Restricting $\zz_{\textrm{ne}}^{\textrm{def}}$ to only trivially stratified bordisms (meaning that there are no 1- or 0-strata) which all carry the same decoration, one obtains a non-extended closed TQFT. 
As shown in \cite{dkr1107.0495, TQFTlecturenotes} one can construct a pivotal 2-category $\B_{\zz_{\textrm{ne}}^{\textrm{def}}}$ from any defect TQFT $\zz_{\textrm{ne}}^{\textrm{def}}$ (along the same lines as one constructs commutative Frobenius algebras from closed TQFTs). 
In the case of state sum models the 2-category is equivalent to the full subbicategory  $\textrm{ssFrob}_{\Bbbk} \subset \textrm{Alg}_{\Bbbk}$ of separable symmetric Frobenius algebras \cite{dkr1107.0495}, and indeed $\End_{\textrm{ssFrob}_{\Bbbk}}(\Bbbk) \cong \Vectk$ where~$\Bbbk$ is the unit object. 
For A- and B-twisted sigmal models, the bicategories are expected to be that of symplectic manifolds and Lagrangian correspondences \cite{ww0708.2851} and of Calabi-Yau varieties and Fourier-Mukai kernels \cite{cw1007.2679}, respectively; in both cases the point serves as the unit object and its endomorphism category is equivalent to $\Vect_\C^{\Z}$. 
And in the case of Landau-Ginzburg models it should be the bicategory $\LG$ (or its $\Q$-graded version $\LGgr$) of isolated singularties and matrix factorisations \cite{cm1208.1481}. 
These are the ``appropriate'' bicategories we have in mind -- if they admit a symmetric monoidal structure (as expected). 

We stress that defect TQFT here only serves as a motivation to consider the bicategories above, and we will not mention defects again. 
A key point is that by choosing bicategories other than $\textrm{Alg}_\Bbbk$ as targets for extended TQFTs~$\zz$, one can associate non-separable $\Bbbk$-algebras to~$\zz$, namely what~$\zz$ assigns to the circle and the pair-of-pants. 

\medskip 

In the present paper we make the above precise for Landau-Ginzburg models. 
In Section~\ref{sec:LGmodels} we review the bicategories $\LG$ and $\LGgr$, and we present symmetric monoidal structures for them which on objects reduce to the sum of polynomials; the unit object is the zero polynomial, and its endomorphism categories are equivalent to $\Vectk^{\Z_2}$ and $\Vectk^\Z$, respectively. 
Moreover, we prove that every object in both $\LG$ and $\LGgr$ is fully dualisable (Corollaries~\ref{cor:LGfullydualisable} and~\ref{cor:LGgrfullydual}). 
Careful and lengthy checks that the data we supply satisfy the coherence axioms of symmetric monoidal bicategories are performed in the PhD thesis \cite{FlavioThesis} for the case $\LG$, and we explain how they carry over to $\LGgr$. 

It follows immediately from the cobordism hypothesis that every object in $\LG$ or $\LGgr$ determines an extended framed TQFT (with values in $\LG$ or $\LGgr$), while generically Landau-Ginzburg models cannot be extended to the point with target $\textrm{Alg}_\Bbbk$. 
Hence our results may be the first explicit demonstration of the general principle that the question of whether or not a given non-extended TQFT can be extended depends on the choice of the target for the extended theory. 

\medskip 

To settle the question of extendability also in the oriented case, we use the results of \cite{HSV, HV, Hesse}: a fully dualisable object~$W$ determines an extended oriented TQFT if and only if the Serre automorphism $S_W \colon W \to W$ (see~\eqref{eq:SerreAutom}) is isomorphic to the unit 1-morphism~$I_W$. 

In Section~\ref{subsec:LGoriented}, we show that for a potential $W \in \Bbbk[x_1,\dots,x_n]$ viewed as an object in $\LG$ we have $S_W \cong I_W[n]$ where $[n]$ is the $n$-fold shift functor which satisfies $[2] = [0]$, cf.\ Section~\ref{subsec:LGdef}. 
Since $I_W \ncong I_W[1]$ this implies that~$W$ determines an extended oriented TQFT (cf.\ Proposition~\ref{prop:LGexori})
\be 
\zzwor \colon \Bordor \lra \LG 
\ee 
if and only if~$n$ is even, and we discuss the relation to Serre functors and Calabi-Yau categories in Remark~\ref{rem:dCY}. 

For a quasi-homogeneous potential $W \in \Bbbk[x_1,\dots,x_n]$ viewed as an object in $\LGgr$ we find that $S_W \cong I_W[n-2]\{ \tfrac{1}{3} c(W) \}$, where $c(W)$ is the central charge of~$W$ (see~\eqref{eq:centralcharge}) and $\{-\}$ denotes the shift in $\Q$-degree. 
Hence every potential~$W$ that satisfies the condition $I_W \cong I_W[n-2]\{ \tfrac{1}{3} c(W)\}$ determines an extended oriented TQFT (cf.\ Proposition~\ref{prop:LGgrexori})
\be 
\zz_{W,\textrm{gr}}^{\textrm{or}} \colon \Bordor \lra \LGgr \, . 
\ee 
If the hypersurface $\{ W = 0 \}$ in weighted projective space is a Calabi-Yau variety (equivalently: if $\tfrac{1}{3} c(W) = n-2$) then the trivialisability of~$S_W$ reduces to the $(n-2)$-Calabi-Yau condition $\Sigma^{n-2} \cong \text{Id}$ on the shift functor $\Sigma = [1]\{1\}$ of the triangulated category $\LGgr(0,W)$, as we show in Corollary~\ref{cor:WCY}. 
This is in line with the general discussion in \cite[Sect.\,4.2]{l0905.0465}. 

Finally, we illustrate the combined power of the cobordism hypothesis and the explicit control over the bicategories $\LG$ and $\LGgr$ by computing the actions of our extended TQFTs on various 2-bordisms: the saddle, the cap, the cup, and the pair-of-pants. 
This is done in terms of the explicit adjunction maps of \cite{cm1208.1481}, for which we discuss two applications: 
\begin{itemize}[itemsep=2pt]
\item 
We explain (in Theorems~\ref{thm:JacW1} and~\ref{thm:JacW2}, Remarks~\ref{rem:gradedJacW1} and~\ref{rem:gradedJacW2}) how the non-separable Jacobi algebra and its residue pairing are recovered from the above extended TQFTs associated to a potential~$W$. 
\item 
The ``TQFTs with corners'' constructed by Khovanov and Rozansky in \cite{kr0401268} can be derived (as we do in Example~\ref{ex:KRTQFT}) directly from the cobordism hypothesis as extended TQFTs that assign the potentials $W = x^{N+1}$ to the point, for all $N\in \Z_{\geqslant 2}$. 
\end{itemize}

\subsubsection*{Acknowledgements} 

We thank 
	Ilka Brunner, 
	Domenico Fiorenza, 
	Jan Hesse, 
	Daniel Murfet 
		and 
	Christoph Schweigert 
for helpful discussions.
The work of N.\,C.~is partially supported by a grant from the Simons Foundation and by the stand-alone project P\,27513-N27 of the Austrian Science Fund. 
The work of F.\,M.\,M. was supported by a fellowship from the Peters-Beer Foundation.

\section{Bicategories of Landau-Ginzburg models} 
\label{sec:LGmodels}

In this section we collect the data that endows the bicategory of Landau-Ginzburg models $\LG$ with a symmetric monoidal structure in which every object has a dual and every 1-morphism has left and right adjoints. 
This is done in Sections~\ref{subsec:LGdef}--\ref{subsec:LGduals}. 
In Section~\ref{subsec:gradedLG} we explain how the analogous results hold for the bicategory of graded Landau-Ginzburg models $\LGgr$. 

\medskip 

Our main reference for bicategories, pseudonatural transformations, modifications etc.\ is \cite{benabou} (see \cite{LeinsterBasic2} for a quick reminder). 
Symmetric monoidal bicategories are reviewed in \cite{Gurskibook, spthesis} and \cite[App.\,A.4]{GregorDiss}; duals for objects and adjoints for 1-morphisms are e.\,g.\ reviewed in \cite{Pstragowski, spthesis}.

\subsection[Definition of $\LG$]{Definition of $\boldsymbol{\LG}$}
\label{subsec:LGdef}

Recall from \cite[Sect.\,2.2]{cm1208.1481} that for a fixed field~$\Bbbk$ of characteristic zero,\footnote{In fact we can allow any commutative unital ring~$\Bbbk$ if we generalise the definition of potentials as in \cite[Def.\,2.4]{cm1208.1481}.} the \textsl{bicategory of Landau-Ginzburg models} $\LG$ is defined as follows. 
An \textsl{object} is either the pair $(\Bbbk,0)$ or a pair $(\Bbbk[x_1,\dots,x_n], W)$ where $n\in\Z_{\geqslant 0}$ and $W\in\Bbbk[x_1,\dots,x_n]$ is a \textsl{potential}, i.\,e.\ the \textsl{Jacobi algebra}
\be
\Jac_W = \Bbbk[x_1,\dots,x_n]/(\partial_{x_1}W, \dots, \partial_{x_n}W)
\ee
is finite-dimensional over~$\Bbbk$. 
We often abbreviate lists of variables $(x_1,\ldots,x_n)$ by~$x$, and we often shorten $(\Bbbk[x], W)$ to~$W$. 

For two objects $(\Bbbk[x], W)$ and $(\Bbbk[z], V)$ we have 
\be\label{eq:LGHom}
\LG \big( (\Bbbk[x], W), (\Bbbk[z], V) \big) 
= 
\hmf\big( \Bbbk[x,z], V-W \big)^\oplus 
\ee 
for the \textsl{Hom category}. 
The right-hand side of~\eqref{eq:LGHom} is the idempotent completion of the homotopy category of finite-rank matrix factorisations of the potential $V-W$ over $\Bbbk[x,z]$. 
We denote matrix factorisations of $V-W$ by $(X,d_X)$ (or simply by~$X$ for short), where $X = X^0 \oplus X^1$ is a free $\Z_2$-graded $\Bbbk[x,z]$-module and $d_X \in \End^1_{\Bbbk[x,z]}(X)$ such that $d_X^2 = W\cdot \id_X$. 
The twisted differentials $d_X, d_{X'}$ induce differentials 
\be 
\label{eq:deltadifferential}
\delta_{X,X'} \colon \zeta \lmt d_{X'}\circ\zeta - (-1)^{|\zeta|}\zeta \circ d_X
\ee 
on the modules $\Hom_{\Bbbk[x,z]}(X,X')$, and 2-morphisms in $\LG$ are even cohomology classes with respect to these differentials. 
Finally, the idempotent completion $(-)^\oplus$ in~\eqref{eq:LGHom} is obtained by considering only matrix factorisations which are direct summands (in the homotopy category of all matrix factorisations) of finite-rank matrix factorisations. 
For more details, see \cite[Sect.\,2.2]{cm1208.1481}. 
	
In passing we note that the category $\LG(W,V)$ has a triangulated structure with the \textsl{shift functor} $[1] \colon \LG(W,V) \to \LG(W,V)$ acting on objects as
\be 
\label{eq:shiftfunctor}
[1] \colon \big( X^0 \oplus X^1, d_X \big) \lmt \big( X^1 \oplus X^0, -d_X \big) \, 
\ee 
see e.\,g.\ \cite[Sect.\,2.1]{kst0511155}. 
It follows that 
\be 
\label{eq:2shift}
[2] := [1] \circ [1] = \textrm{Id}_{\LG(W,V)} \, . 
\ee 

\textsl{Horizontal composition} in $\LG$ is given by functors 
\begin{align}
\otimes \colon 
\LG \Big( (\Bbbk[y], W_2), (\Bbbk[z], W_3) \Big)  \times \, & \LG \Big( (\Bbbk[x], W_1), (\Bbbk[y], W_2) \Big)  \nonumber
	\\ 
& \qquad \lra \LG \Big( (\Bbbk[x], W_1), (\Bbbk[z], W_3) \Big) 
\label{eq:horizfunctor}
\end{align}
which  act on 1-morphisms as 
\be\label{eq:horizYX}
(Y,X) \lmt Y \otimes X \equiv 
	\Big( 
		\big( (Y^0 \otimes_{\Bbbk[y]} X^0) \oplus (Y^1 \otimes_{\Bbbk[y]} X^1) \big) 
		\oplus 
		\big( (Y^0 \otimes_{\Bbbk[y]} X^1) \oplus (Y^1 \otimes_{\Bbbk[y]} X^0) \big) 
	\Big)
\ee 
with $d_{Y \otimes X} = d_Y \otimes 1 + 1 \otimes d_X$, and analogously on 2-morphisms. 
It follows from \cite[Sect.\,12]{dm1102.2957} that the right-hand side of~\eqref{eq:horizYX} is indeed a direct summand of a finite-rank matrix factorisation in the homotopy category over $\Bbbk[x,z]$, hence~$\otimes$ is well-defined. 
Moreover, the \textsl{associator} in $\LG$ is induced from the standard associator for modules, and we will suppress it notationally. 

\begin{remark}
\label{rem:extracare}
One technical issue in rigorously exhibiting $\LG$ as a symmetric monoidal bicategory (as summarised in Sections~\ref{subsec:monoidalLG}--\ref{subsec:symmetricLG}) is to establish an effective bookkeeping device that keeps track of how to transform and interpret various mathematical entities. 
Exercising such care already for the functor~$\otimes$ in~\eqref{eq:horizYX} we can write it as $(\iota_{x,z})_* \circ \otimes_{\Bbbk[x,y,z]} \circ ((\iota_{y,z})^* \times (\iota_{x,y})^*)$, where $\iota_{x,z} \colon \Bbbk[x,z] \hookrightarrow \Bbbk[x,y,z]$ etc.\ are the canonical inclusions, while $(-)_*$ and $(-)^*$ denote restriction and extension of scalars, respectively; \cite[Sect.\,2.3--2.4]{FlavioThesis} has more details. 
\end{remark}

For an object $(\Bbbk[x_1,\dots,x_n], W) \in \LG$, its \textsl{unit} 1-morphism is $(I_W, d_{I_W})$ with 
\be 
I_W = \bigwedge \Big( \bigoplus_{i=1}^n \Bbbk[x,x'] \cdot \theta_i \Big) 
\ee 
where $x' \equiv (x'_1,\dots,x'_n)$ is another list of~$n$ variables, $\{ \theta_i \}$ is a chosen $\Bbbk[x,x']$-basis of $\Bbbk[x,x']^{\oplus n}$, and 
\be 
d_{I_W} = \sum_{i=1}^n \Big( \partial_{[i]}^{x',x} W \cdot \theta_i \wedge (-) + (x'_i-x_i) \cdot \theta_i^* \Big) 
\ee 
where 
\be 
\partial_{[i]}^{x',x} W 
= 
\frac{W(x_1,\dots,x_{i-1}, x'_i, \dots x'_n) - W(x_1,\dots,x_i, x'_{i+1}, \dots x'_n)}{x'_i-x_i}
\ee 
and~$\theta_i^*$ is defined by linear extension of $\theta_i^*(\theta_j) = \delta_{i,j}$ and to obey the Leibniz rule with Koszul signs, cf.\ \cite[Sect.\,2.2]{cm1208.1481}. 
In the following we will suppress the symbol~$\wedge$ when writing elements in or operators on~$I_W$. 

Finally, the \textsl{left} and \textsl{right unitors} 
\be\label{eq:lambdarho}
\lambda_X \colon I_V \otimes X \lra X 
\, , \quad 
\rho_X \colon X \otimes I_W \lra X 
\ee 
for $X\in \LG(W,V)$ are defined as projection to $\theta$-degree zero on the units~$I_V$ and $I_W$, respectively; their explicit inverses (in the homotopy category $\LG(W,V)$) were worked out in \cite{cm1208.1481} to act as follows: 
\begin{align}
\lambda^{-1}_X(e_i) 
& = \sum_{l \geqslant 0} \sum_{a_1 < \cdots < a_l} \sum_j
\theta_{a_1} \ldots \theta_{a_l} 
\left\{ \partial^{z',z}_{[a_l]}d_X \ldots \partial^{z',z}_{[a_1]}d_X \right\}_{ji} \otimes e_{j}
\, , \nonumber
\\
\rho^{-1}_X(e_i) 
& = \sum_{l \geqslant 0} \sum_{a_1 < \cdots < a_l} \sum_j (-1)^{\binom{l}{2} + l|e_i|} e_j \otimes \left\{ \partial^{x',x}_{[a_1]}d_X \ldots \partial^{x',x}_{[a_l]}d_X \right\}_{ji} \theta_{a_1} \ldots \theta_{a_l}
\end{align}
where $\{ e_i \}$ is a basis of the module~$X$, and~$d_X$ is identified with the matrix representing it with respect to $\{ e_i \}$. 

\medskip 

In summary, the above structure makes $\LG$ into a bicategory, cf.\ \cite[Prop.\,2.7]{cm1208.1481}. 
Note that in $\LG$ it is straightforward to determine isomorphisms of commutative algebras (see e.\,g.\ \cite{kr0401268})
\be 
\End(I_W) \cong \Jac_W \, . 
\ee

\subsection[Monoidal structure for $\LG$]{Monoidal structure for $\boldsymbol{\LG}$}
\label{subsec:monoidalLG}

Endowing $\LG$ with a monoidal structure involves specifying the following data: 
\begin{itemize}%
[leftmargin=2.7em]
\item[(M1)] 
monoidal product $\boxempty \colon \LG \times \LG \to \LG$,  
\item[(M2)] 
monoidal unit $\mathds{I} \in \LG$, specified by a strict 2-functor $I\colon \mathds{1} \to \LG$, 
\item[(M3)] 
associator $a\colon \boxempty \circ (\boxempty \times \idlg) \to \boxempty \circ (\idlg \times \boxempty)$, which is part of an adjoint equivalence, 
\item[(M4)] 
pentagonator 
$\pi \colon (\idlg \boxempty a) \circ a \circ (a\boxempty \idlg) \to a \circ a$ (using shorthand notation explained below), 
\item[(M5)] 
left and right unitors 
	$l\colon \boxempty \circ (I \times \idlg) \to \idlg$, 
	$r\colon \boxempty \circ (\idlg \times I) \to \idlg$, 
\item[(M6)] 
2-unitors 
	$\lambda' \colon 1 \circ (l \times 1) \to (l * 1) \circ (a * 1)$, 
	$\rho' \colon r \circ 1 \to (1 * (1 \times r)) \circ (a * 1)$, 
	and 
	$\mu' \colon 1 \circ (r\times 1) \to (1 * (1 \times l)) \circ (a * 1)$ 
	(using shorthand notation), 
\end{itemize}
subject to the coherence axioms spelled out e.\,g.\ in \cite[Sect.\,2.3]{spthesis}. 
In this section we provide the above data for $\LG$, which come as no surprise to the expert. 
The coherence axioms are carefully checked in \cite[Ch.\,3]{FlavioThesis}. 

\bigskip 

\noindent 
(M1) 
We start with the \textsl{monoidal product}. 
It is a 2-functor 
\be 
\sta \colon \LG \times \LG \lra \LG 
\ee 
which is basically given by tensoring over~$\Bbbk$ and taking sums of potentials. 
More precisely, according to \cite[Prop.\,3.1.12]{FlavioThesis}, $\sta$ acts as 
\be 
(W,V) \equiv \big( (\Bbbk[x], W), (\Bbbk[z], V) \big) \lmt \big( \Bbbk[x,z], W+V) \big) \equiv W+V 
\ee 
on objects, while the functors on Hom catgories 
\be 
\sta_{(V_1,V_2), (W_1,W_2)} \colon \big( \LG \times \LG \big) \big( (V_1,V_2), (W_1,W_2) \big) \lra \LG\big(V_1+V_2, W_1+W_2\big) 
\ee 
are given by~$\otimes_{\Bbbk}$ (up to a reordering of variables similar to the situation in Remark~\ref{rem:extracare}, see \cite[Def.\,3.1.3]{FlavioThesis}). 
Compatibility with horizontal composition is witnessed by the natural isomorphisms $\sta_{(U_1,U_2), (V_1,V_2), (W_1,W_2)} \colon \otimes \circ (\sta \times \sta) \to \sta \circ \otimes$ whose $((Y_1,Y_2), (X_1,X_2))$-components are given by linearly extending 
\begin{align}
\big( (Y_1\sta Y_2) \otimes (X_1\sta X_2) \big) & \lra (Y_1 \otimes X_1) \sta (Y_2 \otimes X_2) \, , \nonumber
\\ 
(f_1 \otimes_\Bbbk f_2) \otimes (e_1 \otimes_\Bbbk e_2) & \lmt (-1)^{|f_2| \cdot |e_1|} \cdot (f_1 \otimes e_1) \otimes_\Bbbk (f_2 \otimes e_2)
\end{align}
for $\Z_2$-homogeneous module elements $e_1,e_2,f_1,f_2$, and isomorphisms on units 
\be 
\sta_{(W_1,W_2)} \colon I_{W_1+W_2} \lra I_{W_1} \sta I_{W_2} 
\ee are also standard, cf.\ \cite[Lem.\,3.1.4]{FlavioThesis}. 

\bigskip 

\noindent 
(M2) 
The \textsl{unit object} in $\LG$ is 
\be 
\mathds{I} := (\Bbbk,0) \, . 
\ee  
Let~$\mathds{1}$ be the 2-category with a single object~$*$ and only identity 1- and 2-morphisms. 
We define a strict 2-functor $I\colon \mathds{1} \to \LG$ by setting $I(*) = \mathds{I}$. 

\bigskip 

\noindent 
(M3) 
The \textsl{associator} is a pseudonatural transformation 
\be 
a\colon \sta \circ \, (\sta \times \idlg) \lra \sta \circ (\idlg \times \sta) \circ \mathfrak{A} \, . 
\ee 
Here~$\mathfrak{A}$ is the rebracketing 2-functor $(\LG \times \LG) \times \LG \to \LG \times (\LG \times \LG)$, which we usually treat as an identity. 
The 1-morphism components $a_{((U,V),W)}$ and 2-morphism components $a_{((X,Y),Z)}$ of the associator are given by 
\be 
a_{((U,V),W)} = I_{U+V+W} 
\, , \quad 
a_{((X,Y),Z)} = \lambda_{(X\sta Y) \sta Z}^{-1} \circ \mathcal A_{X,Y,Z} \circ \rho_{X\sta (Y\sta Z)} \, , 
\ee 
where $\mathcal A_{X,Y,Z} \colon X \sta (Y \sta Z) \to (X\sta Y) \sta Z$ is the rebracketing isomorphism for $\LG$, while~$\lambda$ and~$\rho$ are the 2-isomorphisms~\eqref{eq:lambdarho}. 

The associator~$a$ and the pseudonatural transformation
\be 
a^- \colon \sta \circ \, (\idlg \times \sta) \circ \mathfrak{A} \lra \sta \circ (\sta \times \idlg)
\ee 
with components $a^-_{((U,V),W)} = I_{U+V+W}$, $a^-_{((X,Y),Z)} = \lambda_{X\sta (Y \sta Z)}^{-1} \circ \mathcal A^{-1}_{X,Y,Z} \circ \rho_{(X\sta Y) \sta Z}$ are part of a biadjoint equivalence, see \cite[Lem.\,3.2.5--3.2.6]{FlavioThesis}. 

\bigskip 

\noindent 
(M4) 
The \textsl{pentagonator} is an invertible modification 
\begin{align} 
	& 
	\pi \colon \big(1_\sta * (1_{\idlg} \times a)\big) \circ \big(a * 1_{\idlg \times \sta \times \idlg}\big) \circ \big(1_\sta * (a \times 1_{\idlg} )\big)
	\nonumber 
	\\
	&\qquad\qquad\qquad\qquad\qquad\qquad
	 \lra \big(a * 1_{\idlg \times \idlg \times \sta}\big) \circ \big(a*1_{\sta \times \idlg \times \idlg}\big)
\end{align} 
where here and below we write vertical and horizontal composition of pseudonatural transformations as~$\circ$ and~$*$, respectively. 
We also typically use shorthand notation for the sources and targets of modifications obtained by whiskering; for example, the pentagonator is then written 
\be 
\pi \colon (\idlg \sta a) \circ a \circ (a\sta \idlg) \lra a \circ a \, . 
\ee 
Its components are 
\be 
\pi_{(((T,U),V),W)} = \lambda_{I_{T+U+V+W} \otimes I_{T+U+V+W}} \circ \big(\sta_{(T,U+V+W)} \otimes 1_{I_{T+U+V+W}} \otimes \sta_{(T+U+V,W)} \big) 
\, . 
\ee 

\bigskip 

\noindent 
(M5) 
The \textsl{left} and \textsl{right (1-morphism) unitors} are pseudonatural transformations 
\be 
l\colon \sta \circ \, (I \times \idlg) \lra \idlg 
\, , \quad 
r\colon \sta \circ \, (\idlg \times I) \lra \idlg
\ee 
whose components are given by 
\be 
	 l_{(*,W)} = I_W = r_{(W,*)} 
\, , \quad 
l_{(1_*, X)} = \lambda^{-1}_X \circ \rho_X = r_{(X,1_*)} \, , 
\ee 
where we identify $\mathds{1}\times \LG \equiv \LG \equiv \LG \times \mathds{1}$ and $I_0 \sta X \equiv X \equiv X \sta I_0$ (see \cite[Lem.\,3.1.8\,\&\,3.2.11\,\&\,3.2.15]{FlavioThesis} for details). 
The unitors $l,r$ are part of biadjoint equivalences $(l,l^-)$, $(r,r^-)$ as explained in \cite[Lem.\,3.2.13--3.2.15]{FlavioThesis}. 

\bigskip 

\noindent 
(M6) 
The \textsl{2-unitors} are invertible modifications 
	$\lambda' \colon 1 \circ (l \times 1) \to (l * 1) \circ (a * 1)$, 
	$\rho' \colon r \circ 1 \to (1 * (1 \times r)) \circ (a * 1)$, 
	and 
	$\mu' \colon 1 \circ (r\times 1) \to (1 * (1 \times l)) \circ (a * 1)$, 
written here in the shorthand notation also employed in (M4) above, 
whose components are 
\begin{align}
\lambda'_{((*,V),W)} & = \lambda^{-1}_{I_{V+W}} \circ \sta^{-1}_{(V,W)} 
\, , \quad 
\rho'_{((V,W),*)} = (\sta_{(V,W)} \otimes 1_{I_{V+W}}) \circ \lambda^{-1}_{I_{V+W}}
\, , \nonumber 
\\ 
\mu'_{((V,*),W)} & = \rho^{-1}_{I_V \sta I_W} 
\, . 
\end{align}

\medskip 

\begin{proposition}
The data (M1)--(M6) endow $\LG$ with a monoidal structure. 
\end{proposition}

\begin{proof}
The straightforward but lengthy check of all coherence axioms is performed to prove Theorem 3.2.18 in \cite[Sect.\,3.1--3.2]{FlavioThesis}. 
\end{proof}

\subsection[Symmetric monoidal structure for $\LG$]{Symmetric monoidal structure for $\boldsymbol{\LG}$}
\label{subsec:symmetricLG}

Endowing the monoidal bicategory $\LG$ with a symmetric braided structure amounts to specifying the following data: 

\begin{itemize}%
[leftmargin=2.4em]
\item[(S1)] 
braiding $b\colon \sta \to \sta \circ \tau$ as part of and adjoint equivalence $(b,b^-)$, where $\tau \colon \LG \times \LG \to \LG \times \LG$ is the strict 2-functor which acts as $(\zeta,\xi) \mapsto (\xi,\zeta)$ on objects, 1- and 2-morphisms, 
\item[(S2)]
syllepsis $\sigma \colon 1_\sta \to b^- \circ b$, 
\item[(S3)]
$R \colon a \circ b \circ a \to (\idlg \sta b) \circ a \circ (b \sta \idlg)$ 
and 
$S \colon a^- \circ b \circ a^- \to (b \sta \idlg) \circ a^- \circ (\idlg \sta b)$, 
\end{itemize}
subject to the coherence axioms spelled out e.\,g.\ in \cite[Sect.\,2.3]{spthesis}. 
In this section we provide the above data which are discussed in detail in \cite[Sect.\,3.3]{FlavioThesis}. 

\bigskip 

\noindent 
(S1) 
The \textsl{braiding} is a pseudonatural transformation 
\be 
b \colon \sta \lra \sta \circ \tau 
\ee 
whose 1-morphism components $b_{(V,W)}$ are given by $I_{V+W}$ (up to a reordering of variables, see \cite[Not.\,3.1.2\,\&\,Lem.\,3.3.5]{FlavioThesis}), while the 2-morphism components 
\be 
b_{(X,Y)} \colon (Y\sta X) \otimes b_{(V_1,V_2)} \lra b_{(W_1,W_2)} \otimes (X\sta Y)
\ee 
are defined in \cite{FlavioThesis} as natural compositions of canonical module isomorphisms and structure maps of the bicategory $\LG$. 
Explicitly, if $\{ e_a \}$ and $\{ f_b \}$ are bases of the underlying modules of~$X$ and~$Y$, respectively, we have  
\be 
b_{(X,Y)} \colon (f_b \otimes e_a ) \otimes \theta_{i_1}^{j_1} \ldots \theta_{i_{m}}^{j_{m}} 
	\lmt 
(-1)^{|e_a| \cdot |f_b|} \delta_{j_1,0} \dots \delta_{j_{m},0} \cdot \lambda_{X\sta Y}^{-1} (e_a \otimes f_b )
\, . 
\ee 

The braiding~$b$ and the pseudonatural transformation 
\be 
b^- \colon \sta \circ \, \tau \lra \sta 
\ee 
with components $b^-_{(V,W)} = b_{(W,V)}$ and $b^-_{(X,Y)} = b_{(Y,X)}$ are part of a biadjoint equivalence, see \cite[Sect.\,3.3.2]{FlavioThesis}. 

\begin{example}
For a potential $W=x^{N+1}$, $N\in\Z_{\geqslant 2}$, the matrix factorisation $b_{(W,W)}$ is precisely what is assigned to a ``virtual crossing'' in the construction of homological $\mathfrak{sl}_N$-tangle invariants of Khovanov and Rozansky \cite{kr0401268} (see the second expression in \cite[Eq.\,(A.9)]{kr0701333}).
\end{example} 

\noindent 
(S2) 
The \textsl{syllepsis} is an invertible modification 
\be 
\sigma \colon 1_\sta \lra b^- \circ b
\ee 
whose components $\sigma_{(V,W)} \colon I_{V+W} \to b^-_{(V,W)} \otimes b_{(V,W)}$ are given by $\lambda_{I_{V+W}}^{-1}$ (up to a reordering of variables and a sign-less swapping of tensor factors, see \cite[Lem.\,3.3.8]{FlavioThesis}). 

\bigskip 

\noindent 
(S3) 
The invertible modifications 
\begin{align}
& R \colon a \circ b \circ a \lra (\idlg \sta b) \circ a \circ (b \sta \idlg)
\, , \nonumber
\\
& S \colon a^- \circ b \circ a^- \lra (b \sta \idlg) \circ a^- \circ (\idlg \sta b)
\end{align}
have components $R_{((U,V),W)}$ and $S_{((U,V),W)}$ which act on basis elements, i.\,e.\ on tensor and wedge products of $\theta$-variables, by a reordering with appropriate signs, see \cite[Lem.\,3.3.11]{FlavioThesis} for the lengthy explicit expressions. 

\begin{theorem} 
The data (M1)--(M6) and (S1)--(S3) endow $\LG$ with a symmetric monoidal structure. 
\end{theorem}

\begin{proof}
It is shown in \cite[Sect.\,3.1\,\&\,3.3]{FlavioThesis} that the data (S1)--(S3) are well-defined and satisfy the coherence axioms for symmetric braidings. 
\end{proof}

\medskip

We note that instead of directly constructing the data (M1)--(M6) and (S1)--(S3) and verifying their coherence axioms, one could also employ Shulman's method of constructing symmetric monoidal bicategories from symmetric monoidal double categories \cite{Shulman}. 
A double category of Landau-Ginzburg models was first studied in \cite{McNameeThesis}.

\subsection[Duality in $\LG$]{Duality in $\boldsymbol{\LG}$} 
\label{subsec:LGduals}

\subsubsection{Adjoints for 1-morphisms}

Endowing $\LG$ with left and right adjoints for 1-morphisms amounts to specifying the following data: 

\begin{itemize}%
[leftmargin=2.6em]
\item[(A1)] 
1-morphisms $\dX, X^\dagger \in \LG(V,W)$ for every $X \in \LG(W,V)$, 
\item[(A2)]
2-morphisms $\ev_X \colon \dX \otimes X \to I_W$, $\coev_X \colon I_V \to X \otimes \dX$, $\tev_X \colon X \otimes X^\dagger \to I_V$ and $\tcoev_X \colon I_W \to X^\dagger \otimes X$ for every $X \in \LG(W,V)$, 
\end{itemize}
subject to coherence axioms. 
In this section we recall the above data as constructed in \cite{cm1208.1481} (this reference also spells out the coherence axioms). 

\bigskip 

\noindent 
(A1) 
Setting $X^\vee = \Hom_{\Bbbk[x,z]}(X,\Bbbk[x,z])$ and defining the associated twisted differential by $d_{X^\vee}(\phi) = (-1)^{|\phi| + 1} \phi \circ d_X$ for homogeneous $\phi \in X^\vee$, the \textsl{left} and \textsl{right adjoints} of 
\be 
X\in \LG\big((\Bbbk[x_1,\dots,x_n], W), (\Bbbk[z_1,\dots,z_m],V)\big)
\ee 
are given by 
\be 
\label{eq:Xdual}
\dX = X^\vee[m]
\quad \text{and} \quad 
X^\dagger = X^\vee[n] \, , 
\ee 
respectively, where $[m]$ is the $m$-th power of the shift functor $[1]$ in \eqref{eq:shiftfunctor} with itself. 
Hence if in a chosen basis~$d_X$ is represented by the block matrix 
$
(\begin{smallmatrix}
0 & D_1 \\ 
D_0 & 0
\end{smallmatrix})
$, 
then in the dual basis~$d_{\dX}$ is represented by 
$
(\begin{smallmatrix}
0 & D_0^{\textrm{T}} \\ 
-D_1^{\textrm{T}} & 0
\end{smallmatrix})
$ 
if~$m$ is even, and by 
$
(\begin{smallmatrix}
0 & D_1^{\textrm{T}} \\ 
-D_0^{\textrm{T}} & 0
\end{smallmatrix})
$ 
if~$m$ is odd, and similarly for~$d_{X^\dagger}$. 
It follows that $\dX \cong X^\dagger$ if $m=n \,\textrm{mod}\,2$. 

\bigskip 

\noindent 
(A2) 
To present the \textsl{adjunction 2-morphisms}
\begin{align}
& \ev_X \colon \dX \otimes X \lra I_W 
\, , \quad 
\coev_X \colon I_V \lra X \otimes \dX 
\, , \nonumber 
\\
& \tev_X \colon X \otimes X^\dagger \lra I_V 
\, , \quad 
\tcoev_X \colon I_W \lra X^\dagger \otimes X 
\, , \label{eq:4adjumaps}
\end{align}
recall from \cite{Lipman} the basic properties of residues (collected for our purposes in \cite[Sect.\,2.4]{cm1208.1481}), let $\{ e_i \}$ be a basis of~$X$, and define $\Lambda^{(x)} = (-1)^n \partial_{x_1}d_X\ldots \partial_{x_n}d_X$, $\Lambda^{(z)} = \partial_{z_1}d_X\ldots \partial_{z_m}d_X$. 
In \cite{cm1208.1481} the theory of homological perturbation and associative Atiyah classes were used to obtain the following explicit expressions:
\begin{align}
\ev_X( e_i^* \otimes e_j )
&  = \sum_{l \geqslant 0} \sum_{a_1 < \cdots < a_l} (-1)^{\binom{l}{2}+l|e_j|} \, \theta_{a_1} \ldots \theta_{a_l} 
\nonumber 
\\
& \qquad \cdot \Res \left[ \frac{ \big\{ \Lambda^{(z)} \, \partial^{x,x'}_{[a_1]} d_X  \ldots \partial^{x,x'}_{[a_l]} d_X \big\}_{ij}  \, \operatorname{d}\!z }{\partial_{z_1}V, \ldots, \partial_{z_m} V} \right] , 
\nonumber 
\\
\widetilde\ev_X( e_j \otimes e_i^* ) 
& = \sum_{l \geqslant 0} \sum_{a_1 < \cdots < a_l} (-1)^{l + (n+1)|e_j|} \, \theta_{a_1} \ldots \theta_{a_l} 
\nonumber 
\\
& \qquad \cdot \Res \left[ \frac{ \big\{ \partial^{z,z'}_{[a_l]} d_X  \ldots \partial^{z,z'}_{[a_1]} d_X \, \Lambda^{(x)} \big\}_{ij} \, \operatorname{d}\!x }{\partial_{x_1}W, \ldots, \partial_{x_n} W} \right] , 
\nonumber 
\\
\coev_X(\gamma) 
& = \sum_{i,j} (-1)^{\binom{r+1}{2} + mr + s_m} \left\{ \partial^{z,z'}_{[b_1]}d_X \ldots \partial^{z,z'}_{[b_r]}d_X \right\}_{ij} e_{i} \otimes e_j^* \, , 
\nonumber 
\\
\widetilde\coev_X( \bar\gamma ) 
& = \sum_{i,j} (-1)^{(\bar r+1)|e_j| + s_n} \left\{ \partial^{x,x'}_{[\bar b_{\bar r}]}(d_X) \ldots \partial^{x,x'}_{[\bar b_1]}(d_X) \right\}_{ji}  e_i^* \otimes e_j 
\label{eq:evalcoevX}
\end{align}
where $b_i, \bar b_{\bar\jmath}$ and $s_m, s_n\in\Z_2$ are uniquely determined by requiring that $b_1 < \cdots < b_r$, $\bar b_1 < \cdots < \bar b_{\bar r}$, as well as $\bar{\gamma} \theta_{\bar b_1} \ldots \theta_{\bar b_{\bar r}} = (-1)^{s_n} \theta_1 \ldots \theta_n$ and $\gamma \theta_{b_1} \ldots \theta_{b_r} = (-1)^{s_m} \theta_1 \ldots \theta_m$. 

\begin{theorem}
\label{thm:LGsymmetricmonoidal} 
The data (A1)--(A2) endow the bicategory $\LG$ with left and right adjoints for every 1-morphism. 
\end{theorem}

\begin{proof}
This is \cite[Thm.\,6.11]{cm1208.1481}. 
(In fact $\LG$ even has a ``graded pivotal'' structure, see \cite[Sect.\,7]{cm1208.1481}.)
\end{proof}

\subsubsection{Duals for objects} 

Endowing the symmetric monoidal bicategory $\LG$ with duals for objects amounts to specifying the following data: 

\begin{itemize}%
[leftmargin=2.6em]
\item[(D1)] 
an object $W^* \equiv (\Bbbk[x], W)^* \in \LG$ for every $W\equiv (\Bbbk[x],W) \in \LG$, 
\item[(D2)] 
1-morphisms $\ev_W \colon W^* \sta W \to \mathds{I}$ and $\coev_W \colon \mathds{I} \to W \sta W^*$ such that there are 2-isomorphisms 
\begin{align}
& c_{\textrm{l}} \colon r_{(W,*)} \otimes (I_W \sta \ev_W) \otimes a_{((W,W^*),W)} \otimes (\coev_W \sta I_W) \otimes l_{W}^- \lra I_W
\, , \nonumber
\\ 
& c_{\textrm{r}} \colon l_{(*,W^*)} \otimes (\ev_W \sta I_{W^*}) \otimes a^-_{((W^*,W),W^*)} \otimes (I_{W^*} \sta \coev_W) \otimes r^-_{W^*} \lra I_{W^*} 
\, . \nonumber
\end{align}
\end{itemize}
In this section we provide the above data; the explicit isomorphisms $c_{\textrm{l}}, c_{\textrm{r}}$ are constructed in \cite[Ch.\,4]{FlavioThesis}. 

\bigskip 

\noindent 
(D1) 
The \textsl{dual} of $W \equiv (\Bbbk[x],W)$ is $(\Bbbk[x], -W) \equiv W^* \equiv -W$. 

\bigskip 

\noindent 
(D2) 
The \textsl{adjunction 1-morphisms} exhibiting $-W$ as the (left) dual are the matrix factorisations 
\be 
\ev_W = I_W 
\quad \text{and} \quad 
\coev_W = I_W 
\ee 
of $W(x')-W(x)$, viewed as 1-morphisms $(-W) \sta W \to \mathds{I}$ and $\mathds{I} \to W \sta (-W)$, respectively. 

Note that $-W$ is also the right dual of~$W$, with $\tev_W = I_W$ and $\tcoev_W = I_W$ viewed as 1-morphisms $W\sta (-W) \to \mathds{I}$ and $\mathds{I} \to (-W) \sta W$. 

\begin{proposition}
\label{prop:LGdhasduals}
The data (D1)--(D2) endow the monoidal bicategory $\LG$ with duals for every object. 
\end{proposition}

\begin{proof}
The \textsl{cusp isomorphisms} $c_{\textrm{l}}, c_{\textrm{r}}$ are computed in terms of the unitors $\lambda, \rho$ and canonical swap maps in \cite[Lem.\,4.6]{FlavioThesis}. 
\end{proof}

\medskip 

Recall that an object~$A$ of a symmetric monoidal bicategory~$\B$ is \textsl{fully dualisable} if~$A$ has a dual and if the corresponding adjunction 1-morphisms $\ev_A$, $\coev_A$ themselves have left 
	 and right adjoints, which in turn have left and right adjoints, and so on. 
Hence Proposition~\ref{prop:LGdhasduals} together with Theorem~\ref{thm:LGsymmetricmonoidal} implies: 

\begin{corollary}
\label{cor:LGfullydualisable}
Every object of $\LG$ is fully dualisable. 
\end{corollary}

\subsection{Graded matrix factorisations}
\label{subsec:gradedLG} 

Landau-Ginzburg models with an additional $\Q$- or $\Z$-grading appear naturally as (non-functorial) quantum field theories, in their relation to conformal field theories, as well as in representation theory and algebraic geometry. 
In this section we recall the \textsl{bicategory of graded Landau-Ginzburg models} $\LGgr$ from \cite{cm1208.1481, CRCR, MurfetIPMUNotes} (see also \cite{bfk1105.3177v3}) and observe that it inherits the symmetric monoidal structure from $\LG$. 
Moreover, every object in $\LGgr$ is fully dualisable. 

\bigskip

An object of $\LGgr$ is a pair $(\Bbbk[x_1,\dots,x_n], W)$ where now $\Bbbk[x_1,\dots,x_n]$ is a graded ring by assigning degrees $|x_i| \in \Q_{>0}$ to the variables~$x_i$, and $W \in \Bbbk[x_1,\dots,x_n]$ is either zero or a potential of degree~2. 
The \textsl{central charge} of $W \equiv (\Bbbk[x_1,\dots,x_n], W)$ is the numerical invariant 
\be 
\label{eq:centralcharge}
c(W) = 3 \sum_{i=1}^n \big( 1- |x_i| \big) \, . 
\ee 

A 1-morphism $(\Bbbk[x],W) \to (\Bbbk[z],V)$ in $\LGgr$ is a summand of a finite-rank matrix factorisation $(X,d_X)$ of $V-W$ over $\Bbbk[x,z]$ such that the following four conditions are satisfied: 
(i) 
the modules 
	 $X^0 = \bigoplus_{q\in \Q} X^0_q$ and $X^1= \bigoplus_{q\in \Q} X^1_q$ 
are $\Q$-graded, 
(ii) 
the action of~$x_i$ and~$z_j$ on~$X$ are respectively of $\Q$-degree~$|x_i|$ and~$|z_j|$, 
(iii) 
the map~$d_X$ has $\Q$-degree~1, and 
(iv) 
if we write $\{-\}$ for the shift in $\Q$-degree and if $X^i \cong \bigoplus_{q\in\Q} \Bbbk[x,z]\{q\}^{\oplus a_{i,q}}$ for $i \in \{0,1\}$, then $\{ q\in\Q \,|\, a_{i,q} \neq 0 \}$ must%
\footnote{%
Without condition (iv) we still obtain a bicategory $\LGhuge$, with the same structures that we exhibit here for $\LGgr$. 
As explained in \cite[Lecture 3]{MurfetIPMUNotes}, the Hom categories $\LGhuge(W,V)$ are equivalent to infinite direct sums of $\LGgr(W,V)$ with itself, hence we can restrict to $\LGgr$. 
}
 be a subset of $i + G_{V-W}$, 
where 
\be 
G_{V-W} := \big\langle |x_1|, \dots, |x_n|, |z_1|, \dots, |z_m| \big\rangle \subset \Q 
\quad \text{and} \quad 
G_0 := \Z \, . 
\ee 
A 2-morphism in $\LGgr$ between two 1-morphisms $(X,d_X), (X',d_{X'})$ is a cohomology class of $\Z_2$- and $\Q$-degree~0 with respect to the differential $\delta_{X,X'}$ in \eqref{eq:deltadifferential}. 

We continue to write $[-]$ for the $\Z_2$-grading shift and $\{-\}$ for the $\Q$-grading shift. 
Translating \cite[Thm.\,2.15]{kst0511155} into our conventions we see that $\LGgr(W,V)$ has the structure of a triangulated category with \textsl{shift functor} 
\be 
\label{eq:shiftfunctorgraded}
\Sigma := [1]\{1\} \, . 
\ee 	

Since the categories $\LGgr(W,V)$ are idempotent complete (cf.\ \cite[Lem.\,2.11]{kst0511155}) the construction of \cite{dm1102.2957} ensures that horizontal composition in $\LGgr$ can be defined analogously to \eqref{eq:horizfunctor}. 
Moreover, the units~$I_W$ of $\LG$ can naturally be endowed with an appropriate $\Q$-grading (by setting $|\theta_i| = |x_i| - 1$ and $|\theta_i^*| = 1 - |x_i|$), and the associator~$\alpha$ and unitors $\lambda, \rho$ of $\LGgr$ are those of $\LG$ (as they manifestly have $\Q$-degree~0). 
Hence $\LGgr$ is indeed a bicategory. 

\medskip 

The bicategory $\LGgr$ also inherits a symmetric monoidal structure from $\LG$. 
This is so because all 1- and 2-morphisms in the data (M1)--(M6), (S1)--(S3) are constructed from the units~$I_W$ and from the structure maps $\alpha, \lambda, \rho$, their inverses and ($\Q$-degree~0) swapping maps, respectively. 

\medskip 

For a 1-morphism 
\be 
X\in \LGgr\big((\Bbbk[x_1,\dots,x_n], W), (\Bbbk[z_1,\dots,z_m],V)\big)
\ee 
we define its left and right adjoint as 
\be 
\label{eq:Xgradedadjoints}
\dX = X^\vee[m]\{ \tfrac{1}{3} c(V) \}
\, , \quad 
X^\dagger = X^\vee[n]\{ \tfrac{1}{3} c(W) \} \, . 
\ee 
The above shifts in $\Q$-degree are necessary to render the adjunction maps $\ev_X, \coev_X, \tev_X, \tcoev_X$ in \eqref{eq:4adjumaps} and \eqref{eq:evalcoevX} to be of $\Q$-degree~0 so that they are indeed 2-morphisms in $\LGgr$. 
Finally, the (left and right) dual of $(\Bbbk[x], W) \in \LGgr$ is $(\Bbbk[x], -W)$ with the same grading, and the matrix factorisation underlying the adjunction 1-morphisms $\ev_W, \coev_W, \tev_W, \tcoev_W$ is again~$I_W$, but now viewed as a $\Q$-graded matrix factorisation. 

\medskip 

In summary, we have: 

\begin{theorem}
The bicategory $\LGgr$ inherits a symmetric monoidal structure from $\LG$, every object of $\LGgr$ has a dual, and every 1-morphism has adjoints. 
\end{theorem}
 
\begin{corollary}
\label{cor:LGgrfullydual}
Every object of $\LGgr$ is fully dualisable. 
\end{corollary}

\section{Extended TQFTs with values in $\boldsymbol{\LG}$ and $\boldsymbol{\LGgr}$}
\label{sec:extendedTQFTs}

In this section we study extended TQFTs with values in $\LG$ and $\LGgr$. 
We briefly review framed and oriented 2-1-0-extended TQFTs and their ``classification'' in terms of fully dualisable objects and trivialisable Serre automorphisms, respectively. 
Then we observe that every object $W \equiv (\Bbbk[x_1,\dots,x_n],W)$ in $\LG$ or $\LGgr$ gives rise to an extended framed TQFT (Proposition~\ref{prop:Wfullydual} and Remark~\ref{rem:gradedJacW1}), and we show precisely when~$W$ determines an oriented theory (Propositions~\ref{prop:LGexori} and~\ref{prop:LGgrexori}). 
We also show how the extended framed (or oriented) TQFTs recover the Jacobi algebras $\Jac_W$ as commutative (Frobenius) $\Bbbk$-algebras (Theorems~\ref{thm:JacW1} and~\ref{thm:JacW2}, Remark~\ref{rem:gradedJacW2}), and we explain how a construction of Khovanov and Rozansky can be recovered as a special case of the cobordism hypothesis (Example \ref{ex:KRTQFT}).

\subsection{Framed case}
\label{subsec:LGframed}

Recall from \cite[Sect.\,3.2]{spthesis} and \cite[Sect.\,5]{Pstragowski} that there is a symmetric monoidal bicategory $\Bordfr$ of framed 2-bordisms. 
Its objects, 1- and 2-morphisms are, roughly, disjoint unions of 2-framed points~$+$ and~$-$, 2-framed 1-manifolds with boundary and (equivalence classed of) 2-framed 2-manifolds with corners. 
For any symmetric monoidal bicategory~$\B$, the cobordism hypothesis, originally due to \cite{BDpaper}, describes the 2-groupoid $\textrm{Fun}_{\textrm{sym}\,\otimes}(\Bordfr, \B)$ (of symmetric monoidal 2-functors $\zz \colon \Bordfr \to \B$, their symmetric monoidal pseudonatural transformations and modifications) in terms of data internal to~$\B$ that satisfy certain finiteness conditions. 
Objects of $\textrm{Fun}_{\textrm{sym}\,\otimes}(\Bordfr, \B)$ are called \textsl{extended framed TQFTs} with values in~$\B$. 

To formulate the precise statement of the cobordism hypothesis, denote by~$\Bfd$ the full subbicategory of~$\B$ whose objects are fully dualisable, and write $\mathscr{K}(\Bfd)$ for the \textsl{core} of~$\Bfd$, i.\,e.\ the subbicategory of~$\Bfd$ with the same objects and whose 1- and 2-morphisms are the equivalences and 2-isomorphisms of~$\B$, respectively. 
Then: 

\begin{theorem}[Cobordism hypothesis for framed 2-bordisms, {\cite[Thm.\,8.1]{Pstragowski}}]
\label{thm:CHframed}
Let~$\B$ be a symmetric monoidal bicategory. 
There is an equivalence 
\begin{align}
\textrm{Fun}_{\textrm{sym}\,\otimes} \big( \!\Bordfr, \B \big) & \stackrel{\cong}{\lra} \mathscr{K}(\Bfd) \, , \nonumber
\\ 
\zz & \lmt \zz(+) \, . 
\end{align}
\end{theorem} 

Note that thanks to the description of $\Bordfr$ as a symmetric monoidal bicategory in terms of generators and relations given in \cite{Pstragowski}, the action of~$\zz$ is fully determined (up to coherent isomorphisms) by what it assigns to the point. 
For example, if $\zz(+) = A$, then the 2-framed circle which is the horizontal composite of the two semicircles (or \textsl{elbows}) ${}^\dagger\!\ev_{+}$ and $\ev_{+}$ is sent to $\ev_A \otimes\,{}^\dagger\!\ev_A$. 
Similarly, 2-morphisms in $\Bordfr$ can be decomposed into cylinders and adjunction 2-morphisms for $\ev_{+}, \coev_{+}$ and their (multiple) adjoints; we will discuss several examples of such decompositions in the proofs of Theorems~\ref{thm:JacW1} and~\ref{thm:JacW2} below. 

\medskip 

We now turn to the symmetric monoidal bicategory of Landau-Ginzburg models~$\LG$. 
As a direct consequence of the cobordism hypothesis and Corollary~\ref{cor:LGfullydualisable} we have: 

\begin{proposition}
\label{prop:Wfullydual} 
Every object $W \equiv (\Bbbk[x_1,\dots,x_n], W) \in \LG$ determines an extended framed TQFT 
\be 
\zzwfr \colon \Bordfr \lra \LG 
\quad \text{ with } \quad \zzwfr(+) = W \, . 
\ee 
\end{proposition}

This can be interpreted as ``every Landau-Ginzburg model can be extended to the point as a framed TQFT''. 
In the remainder of Section~\ref{subsec:LGframed} we make this more precise by relating $\zzwfr$ to the (non-extended) closed oriented TQFT 
\be 
\zz_W \colon \Bord_{2,1}^{\textrm{or}} \lra \Vectk 
\ee
which via the standard classification in terms of commutative Frobenius algebras (see e.\,g.\ \cite{Kockbook}) is described by the Jacobi algebra $\Jac_W$ with pairing 
\be 
\label{eq:residuepairing}
\langle -,- \rangle_W \colon \Jac_W \otimes_{\Bbbk} \Jac_W \lra \Bbbk 
\ee 
induced by the \textsl{residue trace map} 
\be
\label{eq:residuemap}
\Jac_W \lra \Bbbk 
\, , \quad 
\phi \lmt \Res \left[ \frac{ \phi \, \operatorname{d}\!x }{\partial_{x_1}W, \ldots, \partial_{x_n} W} \right] 
	=: \langle \phi \rangle_W \, , 
\ee 
i.\,e.\ $\langle \phi, \psi \rangle_W = \langle \phi\psi \rangle_W$. 

\medskip 

To recover the $\Bbbk$-algebra $\Jac_W$ with its multiplication $\mu_{\Jac_W} \colon \phi\otimes\psi \mapsto \phi\psi$, we want to show that $\Jac_W$ and $\mu_{\Jac_W}$ are what~$\zzwfr$ assigns to ``the'' circle and ``the'' pair-of-pants. 
However, there are infinitely many isomorphism classes of 2-framed circles (one for every integer), so we have to be more specific. 
Using the equivalent description of 2-framed circles in terms of immersions $\iota \colon S^1 \to \R^2$ together with a normal framing \cite[Sect.\,1.1]{dsps1312.7188}, the correct choice is to take the standard circle embedding for~$\iota$ together with outward pointing normals. 
We denote the corresponding 2-framed circle~$S^1_0$. 
In terms of the structure 1-morphisms of $\Bordfr$ (whose horizontal composition we write as~$\#$), we have (see \cite[Sect.\,1.2]{dsps1312.7188})
\be 
S^1_0 = \ev_+ \# \, {}^\dagger \! \ev_+ \, . 
\ee 
This is the correct choice in the sense that for every integer~$k$, there is a 2-framed circle~$S^1_k$, and for every pair $(k,l) \in \Z^2$ there is a pair-of-pants 2-morphism $S^1_k \sqcup S^1_l \to S^1_{k+l}$ in $\Bordfr$, and only for $k=0=l$ do we get a multiplication. This is ``the'' 2-framed pair-of-pants for us. 

\begin{theorem}
\label{thm:JacW1} 
For every $(\Bbbk[x_1,\dots,x_n], W) \in \LG$, we have that 
\be 
\Big(\zzwfr(S^1_0), \, \zzwfr(\textrm{pair-of-pants}) \Big)  
\ee 
is isomorphic to $(\Jac_W, \mu_{\Jac_W})$ as a $\Bbbk$-algebra. 
\end{theorem}

\begin{proof}
Note first that $\zzwfr(S^1_0) \cong \ev_W \otimes\,{}^\dagger\!\ev_W = \ev_W \otimes (\ev_W^\vee[0]) = I_W \otimes_{\Bbbk[x,x']} I_W^\vee$ is isomorphic in $\LG( (\Bbbk,0), (\Bbbk,0) ) \cong \textrm{vect}^{\Z_2}$ to the vector space $\Jac_W$ (viewed as a $\Z_2$-graded vector space concentrated in even degree). 
One can check that an explicit isomorphism $\kappa \colon I_W \otimes_{\Bbbk[x,x']} I_W^\vee \cong \End_{\LG}(I_W) \cong \Jac_W$ is given by linear extension of $p(x) q(x') \cdot e_i \otimes e_j^* \mapsto p(x)q(x) \cdot \delta_{i,j}$, where~$p$ and~$q$ are polynomials and $\{ e_i \}$ is a basis of the $\Bbbk[x,x']$-module~$I_W$. 

Next we prove that $\zzwfr$ sends the pair-of-pants to the commutative multiplication $\mu_{\Jac_W}$. 
For this we decompose the pair-of-pants into generators, namely into cylinders over the left and right elbows $\ev_{+} \colon - \sqcup \,+ \to \emptyset$ and ${}^\dagger\!\ev_{+} \colon \emptyset \to - \sqcup +$, respectively, and the ``upside-down saddle'' $\ev_{\ev_{+}} \colon {}^\dagger\!\ev_{+} \# \ev_{+} \to 1_{- \sqcup +}$ (which is called~$v_1$ in \cite[Ex.\,1.1.7]{dsps1312.7188}). 
Then
\be 
\text{pair-of-pants} = 1_{\ev_{+}} \,\#\, \ev_{\ev_{+}} \,\#\; 1_{{}^\dagger\!\ev_{+}} 
\colon 
S^1_0 \sqcup S^1_0 \lra S^1_0 \, . 
\ee 
Hence 
	if $\zzwfr(\ev_{\ev_+}) = \ev_{\ev_W}$, then the functor 
$\zzwfr$ sends this pair-of-pants to $1_{\ev_W} \otimes \ev_{\ev_W} \otimes 1_{{}^\dagger\!\ev_W}$, which by pre- and post-composition with the isomorphism $\kappa \colon \ev_W \otimes\,{}^\dagger\!\ev_W \cong \Jac_W$ becomes a map $\mu \colon \Jac_W \otimes_{\Bbbk} \Jac_W \to \Jac_W$. 
Noting that both~$\kappa$ and $\ev_{\ev_W}$ act diagonally (with $\ev_{\ev_W}(e_i^* \otimes e_j) = \delta_{i,j}$ since $\ev_W \colon (-W) \sta W \to \mathds{I}$ has trivial target, cf.\ the explicit expression for $\ev_{\ev_W}$ in~\eqref{eq:evalcoevX}), we find that~$\mu$ is indeed given by multiplication of polynomials, i.\,e.\ $\mu = \mu_{\Jac_W}$. 

	To complete the proof we need to argue that~$\zzwfr$ assigns our choice of counit $\ev_{\ev_W}$ to the upside-down saddle $\ev_{\ev_{+}}$, and not some other choice of adjunction data. 
	By \cite[Thm.\,3.17 \& Thm.\,8.1]{Pstragowski}, extended framed TQFTs are equivalent to ``coherent fully dual pairs'' in their target bicategories, see \cite[Def.\,3.12]{Pstragowski} for the details. 
	For the algebra structure on $\zzwfr(S^1_0)$, we only need a ``coherent dual pair'' as defined in \cite[Def.\,2.6]{Pstragowski}. 
	One straightforwardly checks that $(W,-W,\ev_W,\coev_W, c_{\textrm{r}}, c_{\textrm{l}})$ satisfies all the defining properties of a coherent dual pair, ensuring that $\zzwfr$ can indeed be chosen such that $\zzwfr(\ev_{\ev_+}) = \ev_{\ev_W}$. 
	(The key defining properties of coherent dual pairs for us to check are the so-called swallowtail identities of \cite[Def.\,2.6]{Pstragowski}, which can be viewed as consistency constraints on our cusp isomorphism~$c_{\textrm{l}}$. 
	But since $c_{\textrm{l}}, c_{\textrm{r}}$ and all other 2-morphisms that appear in the swallowtail identities are structure maps of the underlying bicategory of $\LG$, the coherence theorem for bicategories guarantees that the constraints are satisfied.) 
\end{proof}

\begin{remark}
The finite-dimensional $\Bbbk$-algebra $\Jac_W$ is typically not separable. 
For example, if $W=x^{N+1}$ with $N\in \Z_{\geqslant 2}$ the algebra $\Jac_W$ has  non-semisimple representations (as multiplication by~$x$ has non-trivial Jordan blocks) and hence cannot be separable. 
Thus $\Jac_W$ is not fully dualisable in the bicategory $\textrm{Alg}_{\Bbbk}$ of finite-dimensional $\Bbbk$-algebras, bimodules and intertwiners \cite{l0905.0465, spthesis}, so $\Jac_W$ cannot describe an extended TQFT with values in $\textrm{Alg}_{\Bbbk}$.  
Proposition~\ref{prop:Wfullydual} and Theorem~\ref{thm:JacW1} explain how $\Jac_W$ does appear in an extended TQFT with values in $\LG$, namely as the algebra assigned to the circle~$S^1_0$ and its pair-of-pants. 
\end{remark} 

For an algebra $A \in \textrm{Alg}_\Bbbk$ its Hochschild cohomology $\textrm{HH}^\bullet(A)$ is isomorphic to $\ev_A \otimes\,{}^\dagger\!\ev_A$, and for Hochschild homology one finds $\textrm{HH}_\bullet(A) \cong \ev_A \otimes \, b_{(A,A)} \otimes \coev_A$. 
Similarly, for every object $W \equiv (\Bbbk[x_1,\dots,x_n], W) \in \LG$ we may define 
\be 
\textrm{HH}^\bullet(W) := \ev_W \otimes\,{}^\dagger\!\ev_W 
\, , \quad 
\textrm{HH}_\bullet(W) := \ev_W \otimes \, b_{(W,W)} \otimes \coev_W \, . 
\ee 
Thus by Theorem~\ref{thm:JacW1} we have $\textrm{HH}^\bullet(W) \cong \Jac_W$, and paralleling the first part of the proof we find $\textrm{HH}_\bullet(W) \cong \Jac_W[n]$ as $\Z_2$-graded vector spaces (because the matrix factorisations $b_{(W,W)}$ and $\coev_W$ are $I_W \sta I_W$ and $I_W \cong I_W^\dagger = I_W^\vee[n]$, respectively). 
Hence $\textrm{HH}^\bullet(W)$ and $\textrm{HH}_\bullet(W)$ precisely recover the Hochschild cohomology and homology of the 2-periodic differential graded category of matrix factorisations $\textrm{MF}(\Bbbk[x],W)$ as computed in \cite[Cor.\,6.5\,\&\,Thm.\,6.6]{d0904.4713}: 

\begin{corollary}
\label{cor:HH}
For every $W \equiv (\Bbbk[x], W) \in \LG$ we have  
\be 
\textrm{HH}^\bullet(W) \cong \textrm{HH}^\bullet \big(\textrm{MF}(\Bbbk[x],W)\big)
\, , \quad 
\textrm{HH}_\bullet(W) \cong \textrm{HH}_\bullet \big(\textrm{MF}(\Bbbk[x],W)\big) \, . 
\ee 
\end{corollary}

\begin{remark}
\label{rem:gradedJacW1}
Proposition~\ref{prop:Wfullydual}, Theorem~\ref{thm:JacW1} and Corollary~\ref{cor:HH} have direct analogues for the graded Landau-Ginzburg models of Section~\ref{subsec:gradedLG}. 
Firstly, Theorem~\ref{thm:CHframed} and Corollary~\ref{cor:LGgrfullydual} immediately imply that every object $(\Bbbk[x_1,\dots,x_n], W) \in \LGgr$ determines an extended TQFT 
\be
\zz_{W,\textrm{gr}}^{\textrm{fr}} \colon \Bordfr \lra \LGgr \, . 
\ee 

Secondly, going through the proof of Theorem~\ref{thm:JacW1} we see that to the circle~$S^1_0$ and its pair-of-pants, $\zz_{W,\textrm{gr}}^{\textrm{fr}}$ assigns the Jacobi algebra $\Jac_W$ which is now a $\Q$-graded algebra with degree-preserving multiplication. 
We note that here it is important that the upside-down saddle $\ev_{\ev_{+}} \colon {}^\dagger\!\ev_{+} \# \ev_{+} \to 1_{- \sqcup +}$ involves the \textsl{left} adjoint of $\ev_{+}$: by~\eqref{eq:Xgradedadjoints} we have ${}^\dagger\!\ev_{W} = \ev_W^\vee[0]\{0\} = \ev_W^\vee$, so $\zz_{W,\textrm{gr}}^{\textrm{fr}}(\text{pair-of-pants})$ really gives a map 
\be 
\Jac_W \otimes_{\Bbbk} \Jac_W \cong 
	\big( \ev_W \otimes\, {}^\dagger\!\ev_{W}\big) \sta \big( \ev_W \otimes\, {}^\dagger\!\ev_{W}\big)
	\lra 
	\big( \ev_W \otimes\, {}^\dagger\!\ev_{W}\big) \cong \Jac_W . 
\ee 
(Incorrectly using the right adjoint $\ev_W^\dagger = \ev_W^\vee [2n]\{ \tfrac{2}{3} c(W) \}$ would lead to unwanted $\Q$-degree shifts in the multiplication. 
In Remark~\ref{rem:gradedJacW2} below however we are naturally led to use the right adjoint $\ev_W^\dagger$ to obtain the correct graded trace map $\langle-\rangle_W$ on $\Jac_W$.) 

Thirdly, for every $(\Bbbk[x_1,\dots,x_n], W) \in \LGgr$ the matrix factorisation underlying $\coev_W$ is $I_W \cong I_W^\dagger = I_W^\vee[n]\{ \tfrac{1}{3} c(W) \} = I_W^\vee[n-2]\{ \tfrac{1}{3} c(W) \}$, and hence we have 
\be 
\textrm{HH}^\bullet(W) \cong \Jac_W
\, , \quad 
\textrm{HH}_\bullet(W) \cong \Jac_W[n-2]\{ \tfrac{1}{3} c(W) \} \, . 
\ee 
\end{remark}

\subsection{Oriented case}
\label{subsec:LGoriented}

An \textsl{extended oriented TQFT} with values in a symmetric monoidal bicategory~$\B$ is a symmetric monoidal 2-functor $\zz \colon \Bordor \to \B$. 
Here $\Bordor$ is the bicategory of oriented 2-bordisms defined and explicitly constructed in \cite[Ch.\,3]{spthesis}; see in particular Fig.~3.13 of loc.\ cit.\ for a list of the 2-morphism generators (to wit: the saddle, the upside-down saddle, the cap, the cup, and cusp isomorphisms) and their relations. 
Hence objects of $\Bordor$ are disjoint unions 
	 of 
positively and negatively oriented points, which we (also) denote~$+$ and~$-$, respectively. 
It was argued in \cite{l0905.0465} that such 2-functors~$\zz$ are classified by the homotopy fixed points of the $\sotwo$-action induced on $\Bfd$ by the $\sotwo$-action which rotates the framings in $\Bordfr$. 
This was worked out in detail in \cite{HSV, HV, Hesse} as we briefly review next. 

An $\sotwo$-action on $\Bfd$ is a monoidal 2-functor~$\varrho$ from the fundamental 2-groupoid $\Pi_2(\sotwo)$ to the bicategory of autoequivalences of~$\Bfd$. 
Since $\sotwo$ is path-connected, $\Pi_2(\sotwo)$ has essentially a single object~$*$ which~$\varrho$ sends to the identity~$\textrm{Id}_{\Bfd}$ on~$\Bfd$. 
Since $\pi_1(\sotwo) \cong \Z$ the action of~$\varrho$ on 1-morphisms is essentially determined by its value on the identity~$1_*$ corresponding to $1\in\Z$. 
It was argued in \cite[Rem.\,4.2.5]{l0905.0465} that for an oriented TQFT~$\zz$ as above with $\zz(+) =: A$, the relevant choice for $\varrho(1_*)$ is the \textsl{Serre automorphism}~$S_A$ of $A\in \Bfd$. 
By definition~$S_A$ is the 1-morphism 
\be 
\label{eq:SerreAutom} 
S_A :=  
	r_{(A,*)} \otimes \big( 1_A \sta \tev_A \big) \otimes \big( b_{(A,A)} \sta 1_{A^*} \big) \otimes \big( 1_A \sta \tev_A^\dagger \big) \otimes r_A^- 
\colon A \lra A \, . 
\ee 
Here we denote the braiding, horizontal composition and monoidal product in~$\B$ by~$b$, $\otimes$ and~$\sta$, respectively, as we do in $\LG$ and $\LGgr$. 

The bicategory of $\sotwo$-homotopy fixed points $\kbso$ was defined and endowed with a natural symmetric monoidal structure in \cite{HV}. 
Objects of $\kbso$ are pairs $(A,\sigma_A)$ where $A\in\Bfd$ and~$\sigma_A$ is a trivialisation of the Serre automorphism~$S_A$, i.\,e.\ a 2-isomorphism $S_A \to 1_A$ in~$\B$. 
A 1-morphism $(A,\sigma_A) \to (A', \sigma_{A'})$ in $\kbso$ is an equivalence $F \in \Bfd(A,A')$ such that 
$\lambda_F \circ (\sigma_{A'} \otimes 1_F) \circ S_F = \rho_F \circ (1_F \otimes \sigma_A)$
where~$S_F$ is the 2-isomorphism constructed in the proof of \cite[Prop.\,2.8]{HV}, and 2-morphisms $F\to F'$ in $\kbso$ are 2-isomorphism $F\to F'$ in~$\B$. 
Building on \cite{l0905.0465, spthesis, HSV, HV}, extended oriented TQFTs with values in~$\B$ were classified by fully dualisable objects with trivialisable Serre automorphisms in \cite{Hesse}: 

\begin{theorem}[Cobordism hypothesis for oriented 2-bordisms, {\cite[Cor.\,5.9]{Hesse}}]
\label{thm:CHoriented}
Let~$\B$ be a symmetric monoidal bicategory. 
There is an equivalence 
\begin{align}
\textrm{Fun}_{\textrm{sym}\,\otimes} \big( \!\Bordor, \B \big) & \stackrel{\cong}{\lra} \kbso \, , \nonumber
\\ 
\zz & \lmt \zz(+) \, . 
\end{align}
\end{theorem} 

\medskip 

We return to the symmetric monoidal bicategory~$\LG$. 
To determine extended oriented TQFTs with values in $\LG$ we have to compute Serre automorphisms for all objects: 

\begin{lemma}
\label{lem:SWIW} 
Let $W \equiv (\Bbbk[x_1,\dots,x_n], W) \in \LG$. 
Then $S_W \cong I_W[n]$. 
\end{lemma}

\begin{proof}
According to Sections~\ref{subsec:symmetricLG}--\ref{subsec:LGduals}, the factors $r_{(W,*)}$, $1_W$, $\tev_W$, $b_{(W,W)}$, $1_{W^*}$ and $r_W^-$ in the defining expression~\eqref{eq:SerreAutom} are all given by the matrix factorisation underlying the unit $I_W \in \LG(W,W)$, while the matrix factorisation underlying $\tev_W^\dagger = \tev_W^\vee [2n] = \tev_W^\vee$ is $I_W^\vee \cong I_W^\dagger[n] \cong I_W[n]$. 
This leads to $S_W \cong I_W[n]$. 
(A straighforward computation, taking into account subtleties of the kind mentioned in Remark~\ref{rem:extracare}, is carried out in the proof of \cite[Lem.\,5.2.3]{FlavioThesis} to construct an explicit isomorphism $S_W \to I_W$ only in terms of~$\lambda$, $\rho$ and standard swapping isomorphisms.)
\end{proof}

The general fact $I_W \ncong I_W[1]$ (even $\Hom_{\LG(W,W)}(I_W,I_W[1]) = 0$ is true) together with Theorem~\ref{thm:CHoriented} thus imply: 

\begin{proposition}
\label{prop:LGexori} 
An object $W \equiv (\Bbbk[x_1,\dots,x_n], W) \in \LG$ determines an extended oriented TQFT 
\be 
\zzwor \colon \Bordor \lra \LG 
\quad \text{ with } \quad \zzwor(+) = W 
\ee 
if and only if~$n$ is even. 
\end{proposition}

\begin{remark}
\label{rem:dCY}
Let $d\in \Z$. 
Following \cite{KellerCalabiYau}, we say that a $\Bbbk$-linear, Hom-finite triangulated category~$\mathcal T$ with shift functor~$\Sigma$ is \textsl{weakly $d$-Calabi-Yau} if~$\mathcal T$ admits a Serre functor%
\footnote{A \textsl{Serre functor} of~$\mathcal{T}$ is an additive equivalence $\mathcal S_{\mathcal{T}} \colon \mathcal{T} \to \mathcal{T}$ together with isomorphisms $\Hom_{\mathcal T}(A,B) \cong \Hom_{\mathcal T}(B,\mathcal{S}_{\mathcal T}(A) )^*$ that are natural in $A,B \in \mathcal T$.} 
$\mathcal{S}_{\mathcal{T}}$ such that $\Sigma^d \cong \mathcal{S}_{\mathcal{T}}$. 
The triangulated category $\LG(0,W)$ is known to admit a Serre functor $\mathcal{S}_{\LG(0,W)} \cong [n] = [n-2]$. 
Hence $\LG(0,W)$ is weakly $(n-2)$-Calabi-Yau, the Serre automorphism and Serre functor coincide in the sense that $S_W \otimes (-) \cong \mathcal{S}_{\LG(0,W)}$, and the condition that~$S_W$ is trivialisable is equivalent to the condition that the Serre functor is isomorphic to the identity. 
\end{remark}

\begin{remark}
\label{rem:degreebusiness}
\begin{enumerate}
\item 
Proposition~\ref{prop:LGexori} can be interpreted as ``every Landau-Ginzburg model with an even number of variables can be extended to the point as an oriented TQFT''. 
However, since for odd (and even)~$n$ there is an isomorphism of Frobenius algebras 
\begin{align}
\Jac_W & = \Bbbk[x_1,\dots,x_n]/(\partial_{x_1} W, \dots, \partial_{x_n} W) \nonumber 
\\ 
& \cong \Bbbk[x_1,\dots,x_n,y]/(\partial_{x_1} (W+y^2), \dots, \partial_{x_n} (W+y^2), \partial_y (W+y^2)) \nonumber 
\\ 
& = \Jac_{W+y^2} \, , 
\end{align} 
\textsl{every} non-extended oriented Landau-Ginzburg model appears as part of an extended oriented TQFT~$\zzwor$ or $\zz_{W+y^2}^{\textrm{or}}$ (depending on whether~$n$ is even or odd, respectively), namely as the commutative Frobenius algebra with underlying vector space $\zzwor(S^1)$ or $\zz_{W+y^2}^{\textrm{or}}(S^1)$. 
Note that for this argument to work we need to ensure that this Frobenius algebra is really isomorphic to the associated Jacobi algebra, as we do with Theorems~\ref{thm:JacW1} and~\ref{thm:JacW2}. 
\item 
Instead of $\LG$ one can also consider the symmetric monoidal bicategory $\LG^{\bullet/2}$ which is equal to $\LG$ except that the vector space of 2-morphisms $(X,d_X) \to (X,d_{X'})$ is defined to be $H^\bullet_{\delta_{X,X'}}(\Hom_{\Bbbk[x,z]}(X,X'))/\Z_2$, i.\,e.\ both even and odd cohomology of the differential $\delta_{X,X'}$ in \eqref{eq:deltadifferential} are included while $\zeta \in \Hom_{\Bbbk[x,z]}(X,X')$ and~$-\zeta$ are identified after taking cohomology. 
Dividing out this $\Z_2$-action circumvents the issue that without it the interchange law would only hold up to a sign, as we have 
$
(\zeta_1 \otimes \zeta_2) \circ (\xi_1 \otimes \xi_2) 
	= (-1)^{|\zeta_2| \cdot |\xi_1|}
(\zeta_1 \circ \xi_1) \otimes (\zeta_2 \circ \xi_2)
$ 
for appropriately composable homogeneous 2-morphisms. 
Such $\Z_2$-quotients also appear in \cite{kr0401268}; the bicategory $\LG^{\bullet/2}$ is described in more detail in \cite[Sect.\,5.3.1]{FlavioThesis} (where it is denoted LG). 

In particular, for every $(\Bbbk[x_1,\dots,x_n], W) \in \LG^{\bullet/2}$ there is an even/odd isomorphism $I_W \cong I_W[n]$ for~$n$ even/odd. 
Hence by Lemma~\ref{lem:SWIW} every object of $\LG^{\bullet/2}$ determines an extended oriented TQFT with values in $\LG^{\bullet/2}$. 
\item 
A better way to deal with the signs in the interchange law mentioned in part~(ii) above is to incorporate them into a richer conceptual structure. 
Part of this involves the natural differential $\Z_2$-graded categories (with differential~$\delta_{X,X'}$ as above) studied in \cite{d0904.4713}, whose even cohomologies are the matrix factorisation categories of Section~\ref{subsec:LGdef}. 
Such bicategories of differential graded matrix factorisation categories are studied in \cite{bfk1105.3177v3}, and demanding their monoidal product to be made up of differential graded functors produces Koszul signs in the interchange law. 

A wider perspective on Koszul signs and parity issues in Landau-Ginzburg models as discussed here is that they are thought to be the topological twists of supersymmetric quantum field theories, see e.\,g.\ \cite{mirrorbook, ll9112051, hl0404}. 
Formalising this construction in a functorial field theory setting would involve symmetric monoidal super 2-functors on super bicategories of super bordisms, which is a theory whose details to our knowledge have not been worked out. 
Relatedly, we expect the graded pivotal bicategory $\LG$ of \cite{cm1208.1481} to arise as the bicategory associated to a non-extended oriented defect TQFT on super bordisms (which again has not been defined in detail as far as we know), paralleling the non-super construction of \cite{dkr1107.0495} reviewed in \cite{TQFTlecturenotes}. 
\end{enumerate}
\end{remark}

\begin{theorem}
\label{thm:JacW2} 
	 For every $(\Bbbk[x_1,\dots,x_n], W) \in \LG$ with~$n$ even, 
we have that 
\be 
\Big(\zzwfr(S^1_0), \, \zzwfr(\textrm{pair-of-pants}), \, \zzwfr(\textrm{cup})(1), \, \zzwfr(\textrm{cap}) \Big)  
\ee 
is isomorphic to (recall~\eqref{eq:residuemap} for the residue trace $\langle-\rangle_W$)
\be 
	 \Big( \Jac_W, \, \mu_{\Jac_W}, \, 1, \, \langle \zeta(-) \rangle_W \Big)
\ee 
	 as a commutative Frobenius $\Bbbk$-algebra, where the traces $\zzwfr(\textrm{cap})$ and $\langle - \rangle_W$ induce the Frobenius pairings on $\zzwfr(S^1_0)$ and $\Jac_W$, respectively, and $\zeta \in \Jac_W$ is a uniquely determined invertible element. 
\end{theorem}

\begin{proof}
The isomorphism on the level of $\Bbbk$-algebras was already established in Theorem~\ref{thm:JacW1}, it remains to compute the action of~$\zzwfr$ on the cap and cup 2-morphisms. 

The cap is the bordism $\tev_{\ev_+}$ from the 2-framed circle $\ev_+ \# \ev_+^\dagger$
	 to~$1_\emptyset$. We first assume that~$\zzwfr$ sends it
to the 2-morphism $\tev_{\ev_W}$ from $\ev_W \otimes \ev_W^\dagger = \ev_W \otimes \,{}^\dagger\!\ev_W$ to~$I_0$. 
Since $\ev_W \colon (\Bbbk[x],-W) \sta (\Bbbk[x],W) \to (\Bbbk,0)$ has trivial target, only the summand $l=0$ contributes to the expression for $\tev_{\ev_W}$ in~\eqref{eq:evalcoevX}, and pre-composing with the isomorphism $\Jac_W \cong \ev_W \otimes\, {}^\dagger\!\ev_W$ from the proof of Theorem~\ref{thm:JacW1} produces the residue trace $\langle-\rangle_W$. 

Similarly, the $\text{cup} \colon \emptyset \to S^1_0 = \ev_+ \# \,{}^\dagger\!\ev_+$ is equal to 
$
\coev_{\ev_{+}}$. 
Using the explicit expression for $\coev_{\ev_W}$ in~\eqref{eq:evalcoevX} we see that post-composing $\zzwfr(\text{cup})(1)$ with the isomorphism $\ev_W \otimes\,{}^\dagger\!\ev_W \cong \Jac_W$ is indeed the unit $1\in\Jac_W$. 

	To complete the proof we must investigate to what extent our choice of adjunction data in $\LG$ gives rise to a ``coherent fully dual pair'' (where again we rely on the result of~\cite{Pstragowski} that extended framed TQFTs are equivalent to coherent fully dual pairs): if the coherent dual pair $(W,-W,\ev_W,\coev_W, c_{\textrm{r}}, c_{\textrm{l}})$ can be lifted to a coherent fully dual pair then $\zzwfr$ can be chosen such that $\zzwfr(\tev_{\ev_+})=\tev_{\ev_W}$. 
	First we observe that by Lemma~\ref{lem:SWIW} there is a ``fully dual pair'' 
	\be
	\big( W, -W, \ev_W, \coev_W, I_W, I_W, c_{\textrm{r}}, c_{\textrm{l}}, \mu_e, \epsilon_ e, \mu_c, \epsilon_c, \psi, \phi \big) 
	\ee
	in the sense of \cite[Def.\,3.10]{Pstragowski}, where $\phi := \lambda_{I_W} =: \psi$ and $\mu_e, \epsilon_ e, \mu_c, \epsilon_c$ are equal to $\coev_{\ev_W}, \ev_{\ev_W}, \coev_{\coev_W}, \ev_{\coev_W}$ up to appropriate composition with the isomorphisms $\lambda^{\pm 1}, \rho^{\pm 1}$. 
	As explained in the proof of \cite[Thm.\,3.16]{Pstragowski}, every fully dual pair can be made coherent by changing only the counit 2-morphisms by composition with an automorphism~$\zeta$ of~$I_W$ (and possibly the cusp isomorphism~$c_{\textrm{l}}$ which however in our case is not necessary as observed in the proof of Theorem~\ref{thm:JacW1}). 
	Given a fully dual pair, the map~$\zeta$ is uniquely determined by the cusp-counit equation of \cite[Def.\,3.12]{Pstragowski}, which involves two adjunction maps on one side and none on the other. 
	In our case one finds that the constraint reduces to the equality of two linear maps $\Jac_W \otimes_\Bbbk \Jac_W \to \Bbbk$, one of which involves the residue trace $\langle - \rangle_W$ pre-composed with $\zeta \in \textrm{Aut}(I_W) \subset \End(I_W) \cong \Jac_W$, while the other is a composite of structure maps of the symmetric monoidal bicategory $\LG$ (without any adjunction maps). 
\end{proof}

Paralleling the above proof we see that for even~$n$, the extended oriented TQFT $\zzwor$ also assigns the Frobenius algebra $\Jac_W$ to the oriented circle, pair-of-pants, cup and 
	 cap (up to an invertible element $\zeta \in \Jac_W$). 

\begin{example}
\label{ex:KRTQFT}
For every $N\in \Z_{\geqslant 2}$, the potential $x^{N+1}$ determines an extended oriented TQFT with values in the symmetric monoidal bicategory $\LG^{\bullet/2}$ introduced in Remark~\ref{rem:degreebusiness}(ii). 
We denote this TQFT by $\zz_{\textrm{KR}}$ as it recovers -- directly from the cobordism hypothesis -- the explicit construction that Khovanov and Rozansky gave in \cite[Sect.\,9]{kr0401268}. 
In loc.\ cit.\ the authors determine their TQFT by describing what it assigns to the point~$+$, the circle, the cap, the cup and the saddle bordisms in $\Bordor$.
Except for the saddle we have already computed all these assignments of $\zz_{\textrm{KR}}$ for any potential~$W$ in Theorems~\ref{thm:JacW1} and~\ref{thm:JacW2}, and for $W=x^{N+1}$ they match the prescriptions of \cite{kr0401268} (except for non-essential prefactors for the cap and cup morphisms). 

To establish that the TQFT $\zz_{\textrm{KR}}$ indeed matches that of \cite[Sect.\,9]{kr0401268} it remains to compute $\zz_{\textrm{KR}}(\text{saddle}) = \zz_{\textrm{KR}}(\tcoev_{\ev_{+}})$ and compare it to the explicit matrix expressions in \cite[Page\,81]{kr0401268} (or Page~95 of \href{https://arxiv.org/abs/math/0401268v2}{arXiv:math/0401268v2 [math.QA]}). 
Since $\zz_{\textrm{KR}}(\tcoev_{\ev_{+}}) = \tcoev_{\ev_{x^{N+1}}}$ this is another exercise in using the formulas~\eqref{eq:evalcoevX} for adjunction 2-morphisms. 
This is carried out in \cite[Sect.\,5.3.2]{FlavioThesis}, finding 
\be
\label{eq:KRsaddle}
\zz_{\textrm{KR}}(\text{saddle}) 
= 
\begin{pmatrix}
e_{124} & 1 & 0 & 0 \\ 
-e_{234} & 1 & 0 & 0 \\
0 & 0 & -1 & 1 \\ 
 0 & 0 & - e_{234} & - e_{124}
\end{pmatrix}
\ee 
where the entries $e_{ijk} := \sum_{a+b+c=N-1} x_i^a x_j^b x_k^c \in \Bbbk[x_1,x_2,x_3,x_4]$ depend on four variables as the source and target of $\tcoev_{\ev_{x^{N+1}}}$ involve four copies of $x^{N+1} \in \LG^{\bullet/2}$. 
Up to a minor normalisation issue%
\footnote{More precisely, \eqref{eq:KRsaddle} agrees with the saddle morphism of \cite{kr0401268} if the arbitrary polynomial~$r$ of degree $N-2$ in loc.\ cit.\ is set to $\sum_{a+b+c+d=N-2} x_1^a x_2^b x_3^c x_4^d$, and if non-scalar entries of the matrix are multiplied by~$\tfrac{1}{2}$. The latter seems to be a typo in \cite{kr0401268} as without these factors the expression would not be closed with respect to the differential $\delta_{I_{W} \sta I_{-W}, \ev_{x^{N+1}}^\dagger \otimes \ev_{x^{N+1}}}$.} 
the expression~\eqref{eq:KRsaddle} agrees with that of \cite{kr0401268}. 

In summary, we verified that the construction of \cite[Sect.\,9]{kr0401268} can be understood as an application of the cobordism hypothesis to the potential $W=x^{N+1}$. 
\end{example}

\bigskip 

We return to the bicategory $\LGgr$ of Section~\ref{subsec:gradedLG}. 
All the above results in the present section have analogues or refinements in $\LGgr$. 
In particular:  

\begin{proposition}
\label{prop:LGgrexori} 
An object $W \equiv (\Bbbk[x_1,\dots,x_n], W) \in \LGgr$ determines an extended oriented TQFT 
\be 
\label{eq:zzWgror}
\zz_{W,\textrm{gr}}^{\textrm{or}} \colon \Bordor \lra \LGgr 
\quad \text{ with } \quad 
\zz_{W,\textrm{gr}}^{\textrm{or}}(+) = W 
\ee 
if and only if $[n-2]\{ \tfrac{1}{3} c(W) \} \cong \text{Id}_{\LGgr(0,W)}$. 
\end{proposition}

\begin{proof}
By Theorem~\ref{thm:CHoriented}, $W$ determines a TQFT as stated if and only if its Serre automorphism~$S_W$ is trivialisable. 
Paralleling the proof of Lemma~\ref{lem:SWIW} we see that, using~\eqref{eq:Xgradedadjoints}, the matrix factorisation underlying~$S_W$ is 
$\tev_W^\dagger = \tev_W^\vee[2n]\{ \tfrac{2}{3} c(W)\}$.  
Hence~$S_W$ is isomorphic to 
$I_W^\vee[2n]\{ \tfrac{2}{3} c(W)\} = I_W^\dagger[n]\{ \tfrac{1}{3} c(W)\} \cong I_W[n-2]\{ \tfrac{1}{3} c(W)\}$. 
\end{proof}

	Let ${\LGgr}^{/\Z}$ be the symmetric monoidal 2-category obtained from $\LGgr$ by replacing the hom categories $\LGgr(W,V)$ with the orbit categories $\LGgr(W,V)/\Z$ obtained by dividing out the action of the shift functor $\Sigma=[1]\{1\}$, i.\,e.\ 
	\be 
	\Hom_{{\LGgr}^{/\Z}}(X,Y) = \bigoplus_{k\in\Z}\Hom_{\LGgr}(X,\Sigma^k(Y))
	\ee 
	for 1-morphism $X,Y \in \textrm{Ob}({\LGgr}^{/\Z}(W,V)) = \textrm{Ob}({\LGgr}(W,V))$. 
	It follows that in ${\LGgr}^{/\Z}$, we have $X\cong \Sigma^k(X)$ for all 1-morphisms~$X$ and $k\in\Z$ (with~$1_X$ viewed as a 2-isomorphism of degree~$k$). 
	
	In the setting of orbit categories, Calabi-Yau varieties give rise to oriented extended TQFTs: 

\begin{corollary}
\label{cor:WCY}
If for $(\Bbbk[x_1,\dots,x_n], W) \in \LGgr$ the hypersurface $\{ W = 0 \}$ in weighted projective space is a Calabi-Yau variety, then~$W$ determines  
	 an extended oriented TQFT $\Bordor \to {\LGgr}^{/\Z}$. 
\end{corollary}

\begin{proof}
We write $Y_W$ for the zero locus of~$W$ in weighted projective space. 
The variety~$Y_W$ is Calabi-Yau if and only if the condition $c_1(Y_W) = 0$ is satisfied by the first Chern class, which in our normalisation convention is equivalent to $\sum_{i=1}^n |x_i| = |W| = 2$. 
This implies $\tfrac{1}{3} c(W) = \sum_{i=1}^n (1 - |x_i|) = n-2$, and hence according to the proof of Proposition~\ref{prop:LGgrexori} we have that 
\be 
S_W \otimes (-) \cong [n-2]\{n-2\}
\ee 
is the $(n-2)$-fold product of the shift functor $\Sigma = [1]\{1\}$ of $\LGgr(0,W)$ with itself. 
	 Hence $S_W \cong I_W$ in ${\LGgr}^{/\Z}$. 
\end{proof}

\begin{remark}
\label{rem:gradedJacW2}
There is also an analogue of Theorem~\ref{thm:JacW2} for $\zz_{W,\textrm{gr}}^{\textrm{fr}}$: 
We already saw in Remark~\ref{rem:gradedJacW1} that $\zz_{W,\textrm{gr}}^{\textrm{fr}}$ sends the circle and pair-of-pants to $\Jac_W$ as a graded algebra. 
As in the proof of Theorem~\ref{thm:JacW2} we find that $\zz_{W,\textrm{gr}}^{\textrm{fr}}(\text{cup})(1)$ gives the unit $1 \in (\Jac_W)_0$ of degree~0 (because $\coev_{\ev_W}$ is of $\Q$-degree~0). 

Finally, $\zz_{W,\textrm{gr}}^{\textrm{fr}}(\text{cap})$ is a map 
	(up to an invertible element, i.\,e.\ a constant $\zeta \in \Bbbk^\times$) 
from $\ev_W \otimes \ev_W^\dagger = \ev_W \otimes \ev_W^\vee[2n] \{ \tfrac{2}{3} c(W)\} \cong \Jac_W \{ \tfrac{2}{3} c(W)\}$ to~$\Bbbk$. 
This expresses the known fact that the residue trace map $\langle - \rangle_W$ is nonzero only on elements of degree $\tfrac{2}{3} c(W)$. 
For example for $W = x^{N+1}$ we have $\tfrac{2}{3} c(W) = 2(1 - \tfrac{2}{N+1})$ and $\langle x^j \rangle_{x^{N+1}} = \delta_{j,N-1}$, while 
$|x^{N-1}| = (N-1) \tfrac{2}{N+1} = 2 - \tfrac{4}{N+1}$. 
\end{remark}

\end{document}